\documentclass{article}

\usepackage{amsmath}
\usepackage{amssymb}
\usepackage{amsthm}
\usepackage{latexsym}
\usepackage{bbold}
\usepackage{stmaryrd}
\usepackage{dsfont}
\usepackage{hyperref}
\usepackage{enumerate}
\usepackage[noend,boxed,linesnumbered,ruled]{algorithm2e}
\usepackage{graphicx}

\newtheorem{theorem}{Theorem}
\newtheorem{corollary}[theorem]{Corollary}
\newtheorem{lemma}[theorem]{Lemma}
\newtheorem{property}[theorem]{Property}
\newtheorem{problem}[theorem]{Problem}

\newtheorem*{claim}{Claim}
\newtheorem*{nonumbertheorem}{Theorem}
\theoremstyle{definition}

\newtheorem{definition}[theorem]{Definition}
\newtheorem{remark}[theorem]{Remark}

\newcommand{\integers}{\ensuremath{\mathds{Z}}} % the Integers
\newcommand{\naturals}{\ensuremath{\mathds{N}}} % the Natural numbers
\newcommand{\cs}{,\;} % "comma space", a comma followed by some space
\newcommand{\A}{\mathds{A}}
\newcommand{\B}{\mathds{B}}

\newcommand{\piB}{\pi_{1}(\B, u_{0})}
\newcommand{\piA}{\pi_{1}(\A, v_{0})}
\newcommand{\Hzx}{\mathds{B}_{0}^{(x)}} 
\newcommand{\Hx}{\mathds{B}^{(x)}}
\newcommand{\Hnx}[1]{\mathds{B}_{#1}^{(x)}}
\newcommand{\fx}{f^{(x)}}
\newcommand{\ex}{e^{(x)}}
\newcommand{\SC}[3]{\mathcal{SC}({#1},{#2},{#3})}

\newcommand{\tail}[2]{\mathrm{Tail}({#1})} 
\newcommand{\ustart}{u^{\prime}}
\newcommand{\uend}{u^{\prime\prime}}
\newcommand{\Path}[1]{<#1>}

\newcommand{\comment}[1]{}

\numberwithin{figure}{section}

\title{Effective coherence of groups discriminated by a locally quasi-convex hyperbolic group}
\author{Inna Bumagin\footnote{Carleton University, bumagin@math.carleton.ca. Supported by NSERC.}\phantom{i} and 
Jeremy Macdonald\footnote{Stevens Institute of Technology, jeremodm@gmail.com. Partially supported by NSERC.}%\\Version 1.16
}

\begin{document}
\maketitle{}

\begin{abstract}
We prove that every finitely generated group $G$ discriminated by a locally quasi-convex torsion-free hyperbolic group $\Gamma$ is effectively coherent: 
that is, 
presentations for finitely generated subgroups can be computed from the subgroup generators.  We study $G$ via its embedding into an iterated centralizer 
extension of $\Gamma$, and prove that this embedding can be computed.  We also give algorithms to enumerate all finitely 
generated groups discriminated by $\Gamma$ and to decide whether a given group, with decidable word problem, is discriminated by $\Gamma$.
If $\Gamma$ may have torsion, we prove that groups obtained from $\Gamma$ by iterated amalgamated products with virtually abelian groups, over elementary 
subgroups, are effectively coherent.
\end{abstract}

\tableofcontents

\section{Introduction}
A group $G$ is called \emph{coherent} if every finitely generated subgroup is isomorphic to a finitely presented group.
When $G$ is known to be coherent, an algorithmic question immediately arises: can such a presentation be computed from the 
subgroup generators? We say that $G$ is \emph{effectively coherent} if we can describe an algorithm that, given a finite set $X\subset G$, produces 
a presentation for the subgroup $\langle X\rangle$.  
Effective coherence enhances the algorithmic study of subgroups: 
once a subgroup presentation is known one may apply other algorithms that deduce group-theoretic properties from the presentation, or 
solve problems involving several presentations (isomorphism, for example).

%The study of algorithmic problems in combinatorial group theory has produced a great number of algorithms which, starting with a 
%presentation for a group $H$ (and usually some additional knowledge), compute various properties of $H$.  
%The ability to compute presentations for subgroups allows us to use these algorithms in the study of finitely generated subgroups.
%Having a presentation for a 
%subgroup $H\leq G$ allows these tools to be used in studying $H$.

Several classes of groups are known to be effectively coherent, including free groups, limit groups, and locally quasi-convex hyperbolic groups. 
One may drop the assumption that $G$ is coherent and ask whether a subgroup presentation for $\langle X\rangle$ can be computed given the 
guarantee that a finite presentation exists.
Though this problem was shown in \cite{BW11} to be unsolvable in several major classes of groups, some non-coherent examples 
in which it is solvable, for example the direct product of free groups, are known (see \cite{BW11} for a summary).

Our main interest is in groups \emph{discriminated} by another group $\Gamma$.  A group $G$ is discriminated by $\Gamma$ (or is 
\emph{fully residually $\Gamma$}) if for every finite set $\{g_{1},\ldots,g_{n}\}$ of non-trivial elements of $G$ there exists 
a homomorphism $\phi:G\rightarrow \Gamma$ such that $\phi(g_{i})$ is non-trivial for $i=1,\ldots,n$.
When $\Gamma$ is a free group and $G$ is finitely generated, $G$ is called a \emph{limit group}.  Limit groups figured prominently in the
work of Kharlampovich and Miasnikov \cite{KM06} and Sela \cite{Sel06} on the solution to Tarski's problems on the elementary theory 
of free groups. Many equivalent characterizations of limit groups are known, and these characterizations extend beyond the case 
when $\Gamma$ is free (see \cite{KM10} for a summary).

In particular, when $\Gamma$ is torsion-free hyperbolic these characterizations hold.  But not all hyperbolic groups are coherent
(this follows from the well-known construction in \cite{Rip82}), and if $\Gamma$ is not coherent then 
there are non-coherent groups discriminated by $\Gamma$
(indeed, $\Gamma$ itself is an example). We will consider the case when $\Gamma$ is hyperbolic, locally 
quasi-convex, and torsion-free.  Such groups are known to be effectively coherent (see \cite{Gru99Angle} Prop.~6.1).
We prove that every finitely generated group $G$ which is discriminated by $\Gamma$ is also effectively coherent.  

Every such group $G$ is known to embed into a group obtained from $\Gamma$ by iterated centralizer extensions \cite{KM09}, and a partial result 
regarding the computation of this embedding, for the more general case when $\Gamma$ is hyperbolic and torsion-free, was given 
in \cite{KM13}.  We complete this result, giving an algorithm to compute the embedding provided $\Gamma$ is also locally quasi-convex. 
We also give an algorithm that enumerates all (finitely generated) groups discriminated by $\Gamma$, and an algorithm that recognizes 
whether a given group, with decidable word problem, is discriminated by $\Gamma$.  

For the case when $\Gamma$ has torsion, we are not aware of a similar embedding theorem for groups 
discriminated by $\Gamma$.  However, we replace centralizer extensions by amalgamated products with virtually abelian groups, amalgamated over 
elementary subgroups, and prove that groups obtained from $\Gamma$ in this way are effectively coherent.

Our principal results are summarized below.  In the case when $\Gamma$ is a free group, these results were 
obtained in \cite{KM98b} (embedding theorem), \cite{KM06} (effective coherence, see Thm. 30), and \cite{GW09} (effective coherence, enumeration, recognition).

%For limit groups, this embedding was computed in . 
%Effective coherence for limit groups was proved in  (Thm. 30) and .

% as a corollary of an algorithm which enumerates all limit groups.
%We prove these results for $\Gamma$ in the other order: as a corollary of effective coherence for groups discriminated by $\Gamma$, 
%we obtain an algorithm enumerating all 
%(finitely generated) groups discriminated by $\Gamma$, and an algorithm that recognizes whether or not a group is discriminated 
%by $\Gamma$. 

\begin{nonumbertheorem}
Let $\Gamma$ be a hyperbolic group which is torsion-free and locally quasi-convex.  There are algorithms to solve each of the following problems.
\begin{enumerate}[(i)]
\item Given a finitely presented group $G$ known to be discriminated by $\Gamma$, compute a sequence of centralizer extensions of $\Gamma$, 
\[
\Gamma=G_{0}<G_{1}<\ldots<G_{n},
\]
and an embedding of $G$ into $G_{n}$ (Theorem~\ref{Thm:Embedding}).
\item Given a finitely presented group $G$ known to be discriminated by $\Gamma$ and a finite subset $X\subset G$, 
compute a presentation for the subgroup generated 
by $X$ (Theorem~\ref{Thm:EffectiveCoherence}).
\item Given a finitely presented group $G$ and a solution to the word problem in $G$, determine whether or not $G$ is discriminated by $\Gamma$ 
	(Theorem~\ref{Thm:Embedding}).
\item Enumerate, by presentations, all finitely generated groups discriminated by $\Gamma$, without repeating isomorphic groups 
	(Theorem~\ref{Thm:Enumeration}).
\end{enumerate}
In addition, the following problem is algorithmically solvable without the assumption that $\Gamma$ is torsion-free.
\begin{enumerate}[(i)]\setcounter{enumi}{4}
\item Given a sequence of groups \label{Prob:EffectiveCoherenceElementary}
\[
\Gamma=G_{0}<G_{1}<\ldots<G_{n}
\]
in which $G_{i+1}=G_{i}\ast_{E(g_{i})} V_{i}$ with $V_{i}$ virtually abelian and $E(g_{i})$ the maximal elementary subgroup containing the 
infinite-order hyperbolic element $g_{i}$, and a subset $X\subset G_{n}$, compute a presentation for the subgroup generated by $X$ 
(Theorem \ref{Thm:EffectiveCoherenceElementary}).
\end{enumerate}
\end{nonumbertheorem}

Effective coherence in $G$ is the main result, and we prove this by first solving problem (\ref{Prob:EffectiveCoherenceElementary}), which 
gives effective coherence in centralizer extensions in the case when $\Gamma$ is torsion-free
(Theorem~\ref{Thm:EffectiveCoherenceExtensions}).  We view $G_{n}$ as the fundamental group of a graph of groups having 
two vertices, one with vertex 
group $G_{n-1}$ and the other with a (virtually) abelian vertex group, and one edge with edge group being  (virtually) $\integers$. 
A graph-folding algorithm, along with a structure called an $\A$-graph, was developed in \cite{KWM05} 
to find the induced decomposition of, and hence a presentation for, any finitely presented subgroup of the fundamental group of a graph of groups $\A$.  
While certain algorithmic properties are required of $\A$ (it should be `benign', see Definition \ref{Def:Benign}), when the only edge 
group is $\integers$ these properties reduce to the decidability of the following problem in vertex groups.

%Viewing
%$G_{n}$ inductively as an amalgamated product, hence as 
%the fundamental group of a graph of groups, we apply the folding techniques developed in \cite{KWM05} to 
%compute the induced decomposition of, and hence a presentation for, subgroups of $G_{n}$.  In doing so,
%the following algorithmic problem plays a central role.
\begin{problem}
The \emph{power coset membership problem} for a group $G$ asks to decide, given two elements $x,g\in G$ and 
a finitely generated subgroup $H$ of $G$, whether or not there exists a non-zero integer $m$ such that $g^{m}\in xH$.
\end{problem}
Decidability of the power coset membership problem implies decidability of several other important algorithmic problems.  
On input $H=1$ and $g=1$ the answer is `Yes' if and only if $x=1$, so the \emph{word problem} 
is decidable in $G$.  On input $g=1$, the answer is `Yes' if and only if $x\in H$, so the \emph{membership problem} is decidable in $G$.

On input $x=1$, the answer is `Yes' if and only if $H\cap\langle g\rangle$ is non-trivial, and we call this the \emph{power membership problem}. 
Note that $m=0$ is specifically excluded in the problem specification so that this (non-trivial) problem arises when $x=1$. If the intersection is non-trivial, a 
generator for $H\cap \langle g\rangle$ may be produced by finding the smallest $m$ such that $g^{m}\in H$. Finally, we can decide if the intersection 
$xH \cap \langle g\rangle$ is non-empty, since this occurs if and only if either the answer to the power coset membership problem is `Yes' or $x\in H$.
 
To solve the power coset membership problem in $G_{n-1}$, we view $G_{n-1}$ itself as the fundamental group of a two-vertex graph of groups $\A$.
We construct a `folded $\A$-graph' $\Hx$ representing the coset $xH$: doing so requires that $\A$ be benign, and we argue inductively. 
To decide if $g^{m}\in xH$, we develop in \S 2 an algorithm 
\textsc{ReadPower} that determines whether or not a power of $g$ can be `read' in $\Hx$ (Theorem \ref{Thm:ReadingPowers}).  

While reading a fixed power of $g$ in $\Hx$ is straightforward, 
ensuring algorithm termination while trying to read an arbitrary power is difficult.  We provide some general, though somewhat involved, 
conditions under which \textsc{ReadPower} terminates (Property \ref{Property:PowerReadingConditions}), giving a solution to the power coset 
membership problem for certain graphs of groups (Theorem \ref{Thm:PowerCosetMembership}).  As an example, the power coset membership problem is 
decidable whenever all edge groups are finite and, necessarily, the problem is decidable in vertex groups (Corollary~\ref{Cor:FiniteEdgeGroups}).
The proof that $\A$ satisfies these conditions is given in \S 3, and we rely on  the local quasi-convexity of $\Gamma$ and 
local relative quasi-conevxity of $G_{n-1}$.

\section{Power coset membership in graphs of groups}\label{Section:GraphsOfGroups}
When a group $G$ is presented as the fundamental group of a graph of groups, every subgroup 
$H$ inherits from $G$ a graph of groups decomposition, which yields a presentation of $H$.
To compute this presentation,
and to solve the problem of membership in $H$, Kapovich, Weidmann, and Miasnikov developed in \cite{KWM05} a graph-folding algorithm similar to the folding procedure used by Stallings to 
study subgroups of free groups. We apply this technique to study the power coset membership problem in fundamental groups 
of certain graphs of groups.

\subsection{$\A$-graphs and subgroup graphs}\label{Subsec:Agraphs}
We recall from \cite{KWM05} the notion of an $\A$-graph and some important properties of $\A$-graphs.  All of the 
results of \S\ref{Subsec:Agraphs} are from \cite{KWM05}, to which we refer the reader for proofs. 

A \emph{graph} consists of a set $V$ called vertices, 
a set $E$ called edges,  an involution $\phantom{e}^{-1}: E\rightarrow E$ that has no fixed points,  and two functions 
$o:E\rightarrow V$ and $t:E\rightarrow V$ that satisfy $o(e)=t(e^{-1})$ for all 
$e\in E$.

A \emph{graph of groups} $\A$ consists of a graph $A$ together with, for each vertex $v\in V$ a group $A_{v}$, for each edge 
$e\in E$ a group $A_{e}$, and 
for each edge group $A_{e}$ two monomorphisms $\alpha_{e}:A_{e}\rightarrow A_{o(e)}$ and $\omega_{e}: A_{e}\rightarrow A_{t(e)}$.  For inverse edges $e^{-1}$ we
insist that $A_{e^{-1}}=A_{e}$, $\alpha_{e^{-1}}=\omega_{e}$.

An \emph{$\A$-path} $p=\Path{p_{0}, e_{1}, p_{1}, \ldots, e_{n},p_{n}}$ from vertex $v_{0}$ to $v_{n}$ consists of an underlying edge-path 
$(e_{1},\ldots,e_{n})$ from $v_{0}$ to 
$v_{n}$ and a choice of vertex group elements $p_{i}\in A_{t(e_{i})}$ for $i=1,\ldots,n$ and  $p_{0}\in A_{v_{0}}$.  The \emph{length} of $p$ is the number $n$ of edges.  

An \emph{elementary reduction} replaces a subpath 
of the form 
\[
\Path{a,e,\omega_{e}(c),e^{-1},b}, 
\]
where $c\in A_{e}$, with the subpath $\Path{a\alpha_{e}(c)b}$.   Elementary reductions and their inverses generate an equivalence relation 
on $\A$-paths, with the equivalence class of $p$ denoted $\overline{p}$.  A path is \emph{reduced} if no elementary reduction is applicable.

%The path $p$ is equivalent to every path 
%\[
%p' = p_{0}\alpha_{1}(c_{1}), e_{1}, \omega_{1}(c_{1}^{-1})p_{2}\alpha_{2}(c_{2}),\ldots,\omega_{n-1}(c_{n-1}^{-1})p_{n-1}\alpha_{n}(c_{n}), 
%e_{n}, \omega_{n}(c_{n}^{-1})p_{n}
%p' = p_{0}\alpha_{1}(c_{1}), e_{1}, \omega_{e_{1}}(c_{1}^{-1})p_{2}\alpha_{e_{2}}(c_{2}),\ldots,\omega_{e_{n-1}}(c_{n-1}^{-1})p_{n-1}\alpha_{e_{n}}(c_{n}), 
%e_{n}, \omega_{e_{n}}(c_{n}^{-1})p_{n}
%\]
%where $\alpha_{i}$, $\omega_{i}$ denote $\alpha_{e_{i}}$, $\omega_{e_{i}}$ and $c_{i}\in A_{e_{i}}$ for $i=1,\ldots,n$.  
The set of equivalence classes of $\A$-loops based at 
a given vertex $v_{0}$ forms a group under concatenation and is called the \emph{fundamental group of the graph of groups $\A$} and is denoted $\piA$.

\subsubsection*{$\A$-graphs}
Let $\A$ be a graph of groups with underlying graph $A$ and base vertex $v_{0}$.  An $\A$-graph $\mathcal{B}$ consists of a graph $B$ together with the following additional data:
\begin{enumerate}[(i)]
\item a graph morphism $[\cdot]: B\rightarrow A$;
\item for each vertex $u$ of $B$ a subgroup $B_{u}\leq A_{[u]}$;
\item for each edge $f$ of $B$ two group elements $f_{\alpha}\in A_{[o(f)]}$ and $f_{\omega}\in A_{[t(f)]}$ such that $(f^{-1})_{\alpha}=(f_{\omega})^{-1}$.
\end{enumerate}
The $\A$-vertex $[u]$ is called the \emph{type} of a vertex $u\in B$, the $\A$-edge $[f]$ is the type the edge $f$ of $B$,
and we say that $f$ 
has the \emph{label} $(f_{\alpha}, [f], f_{\omega})$.  

The $\A$-graph $\mathcal{B}$ defines a graph of groups $\B$ as follows. Vertex groups are the groups 
$B_{u}$ above, for each edge $f$ the edge group $B_{f}$ is the subgroup of $A_{[f]}$ defined by
\[
B_{f} = \alpha_{[f]}^{-1}(f_{\alpha}^{-1} B_{o(f)} f_{\alpha})\cap \omega_{[f]}^{-1}(f_{\omega} B_{t(f)} f_{\omega}^{-1}),
\]
and the monomorphism $\alpha_{f}: B_{f}\rightarrow B_{o(f)}$ is defined by
\[
\alpha_{f}(g) = f_{\alpha} \alpha_{[f]}(g) f_{\alpha}^{-1}.
\]
Since $\omega_{f}=\alpha_{f^{-1}}$, we have that $\omega_{f}: B_{f}\rightarrow B_{t(f)}$ is given by  
\[
\omega_{f}(g) = f_{\omega}^{-1} \omega_{[f]}(g) f_{\omega}.
\]
Henceforth we will not distinguish between the $\A$-graph $\mathcal{B}$ and its associated graph of groups $\B$, referring to both as $\B$ and saying 
that `$\B$ is an $\A$-graph'.

To each $\B$-path $q=\Path{q_{0}, f_{1}, q_{1}, \ldots, f_{m}, q_{m}}$ from $u_{0}$ to $u_{1}$ we associate an $\A$-path $\mu(q)$ from $[u_{0}]$ 
to $[u_{1}]$ defined by
\[
\mu(q) = \Path{q_{0}(f_{1})_{\alpha}, [f_{1}], (f_{1})_{\omega} q_{1}  (f_{2})_{\alpha}, [f_{2}], \ldots, f_{m}, (f_{m})_{\omega}  q_{m}}.
\]
Note that if path $p$ ends at vertex $u$ and path $q$ begins at $u$, then for the concatenated path $pq$ we have 
\[
\mu(pq) = \mu(p)\mu(q).
\]
For two equivalent $\B$-paths $q$ and $q'$, their corresponding $\A$-paths $\mu(q)$ and $\mu(q')$ are equivalent.  Let 
\[
\overline{L(\B, u_{0}, u_{1})} = \{ \overline{\mu(q)} \; | \; \mbox{$q$ is a reduced $\B$-path from $u_{0}$ to $u_{1}$}\},
\]
where $\overline{\mu(q)}$ denotes the $\A$-equivalence class of $\mu(q)$.
If $u_{0}$ is a vertex of $\B$ such that $[u_{0}]=v_{0}$, then 
$\mu$ induces a homomorphism $\nu:\piB\rightarrow\piA$ whose image is precisely $\overline{L(\B, u_{0}, u_{0})}$.

We are interested in the case when the homomorphism $\nu$ is injective, which occurs when $\B$ is \emph{folded}.  
An $\A$-graph $\B$ is said to be \emph{not folded} if at least one of the following conditions holds.
\begin{enumerate}[(I)]
\item There exists a vertex $u$ and two distinct edges $f_{1}$ and $f_{2}$ with $u=t(f_{1})=o(f_{2})$ and $[f_{1}]=[f_{2}^{-1}]=e$, such that 
\[
(f_{1})_{\omega} b (f_{2})_{\alpha} = \omega_{e}(c)
\]
for some $b\in B_{u}$ and $c\in A_{e}$.
\item There is an edge $f$ such that 
\[
\alpha_{e}^{-1}(f_{\alpha}^{-1} B_{o(f)} f_{\alpha})\neq \omega_{e}^{-1}(f_{\omega} B_{t(f)}f_{\omega}^{-1}),
\]
where $e=[f]$.\label{Def:NotFoldedII}
\end{enumerate}
If neither holds, $\B$ is folded. Note in (I) that an edge $f$ and its inverse $f^{-1}$ are distinct edges.

\begin{lemma} Suppose $\B$ is folded.  Then for every reduced $\B$-path $q$ the $\A$-path $\mu(q)$ is reduced, hence the 
homomorphism $\nu: \piB\rightarrow \piA$ is injective.
\end{lemma} 

Starting with an $\A$-graph $\B'$ that is not folded, there are six \emph{folding moves} F1-F6 and three auxiliary moves A0-A2 that may be applied 
to eliminate instances of (I) and (II) while preserving the image 
of $\piB$ in $\piA$.  

The folding algorithm (Proposition~5.4 of \cite{KWM05}) consists of performing a sequence of folding moves (in any order) until a folded 
graph is obtained.
To carry out the folding moves effectively, and to ensure the existence of a terminating sequence, the following 
conditions are sufficient.

\begin{definition}\label{Def:Benign}
A finite connected graph of finitely generated groups $\mathds{A}$ is said to be \emph{benign} if all of the following conditions are satisfied.
\begin{enumerate}[(1)]
\item \label{BenignCoset} For each vertex $v$ and edge $e$ with $o(e)=v$ there is an algorithm that, given a finite set $X\subset A_{v}$ and an element $a\in A_{v}$ decides 
whether or not $\langle X\rangle \cap a\alpha_{e}(A_{e})$ is empty and if non-empty produces an element of this intersection.
\item \label{BenignNoetherian} Every edge group is Noetherian (i.e. all subgroups are finitely generated).
\item \label{BenignEdgeMembership} Every edge group has decidable uniform membership problem.  That is, there is an algorithm that, given a finite set $X\subset A_{e}$ and an element $a\in A_{e}$, 
decides whether or not $a\in \langle X\rangle$.
\item \label{BenignGenSet} For each vertex $v$ and edge $e$ with $o(e)=v$ there is an algorithm that, given a finite set $X\subset A_{v}$ computes a generating set 
for $\langle X\rangle \cap \alpha_{e}(A_{e})$.
\end{enumerate}
\end{definition}

The main theorem of \cite{KWM05} produces, from a set of subgroup generators, a folded $\A$-graph representing the subgroup and a presentation for the subgroup.

\begin{theorem}\label{Thm:Folding}
Let $\A$ be a benign graph of groups with base vertex $v_{0}$.  
\begin{enumerate}[(i)]
\item There is an algorithm that, given a finite subset $X$ of $\piA$, constructs a folded $\A$-graph $\B$ with base vertex $u_{0}$ such that 
$\overline{L(\B,u_{0},u_{0})}=\langle X\rangle$.  Each vertex group in $\B$ is described by a generating set of elements of vertex groups of $\A$.
\item If each vertex group of $\A$ is effectively coherent then $\piA$ is effectively coherent.
\end{enumerate}
\end{theorem}

\subsection{Cosets of a subgroup}\label{Section:Cosets}

In order to solve the power coset membership problem in $\piA$, 
we will construct a folded graph $\Hx$ similar to the graph $\B$ constructed in Theorem~\ref{Thm:Folding}.  While in $\B$ elements of 
the subgroup $H=\langle X\rangle$ are represented by loops at 
a vertex $u_{0}$, 
elements of $xH$ are represented by paths in $\Hx$ from a distinguished vertex $u_{x}$ to $u_{0}$.  The graph $\Hx$ is not strictly necessary to solve 
power coset membership: 
to decide if $p^{n}\in xH$, we may instead check if $x^{-1}p^{n}$ is represented by a loop in $\mathds{B}$. However, this requires that 
the path $x^{-1}p^{n}$ be reduced, and since the reduced path depends on $n$, this introduces complications. 
The use of $\Hx$ provides a more elegant solution.

\begin{theorem}\label{Thm:CosetGraph}
Let $G=\piA$, where $\A$ is a benign graph of groups.  There is an algorithm that, given a finitely generated subgroup 
$H\leq G$ and $x\in G$, produces either
\begin{enumerate}[(1)]
\item an element $y\in G$ with path length 0 such that $yH=xH$, if such $y$ exists, or
\item a folded $\A$-graph $\Hx$ with distinguished vertices $u_{x}$ and $u_{0}$ such that 
\[
\overline{L(\Hx, u_{x}, u_{0})}=xH.
\]
\end{enumerate}
The analogous result for right cosets holds.
\end{theorem}

\begin{proof}
Let $X$ be the given generating set of $H$.
Construct an $\A$-graph $\Hzx$ as follows.  Begin with a base vertex $u_{0}$. 
We may assume that each generator $h\in X$ is given as reduced $\A$-path
\[
p_{h} = \Path{h_{0}, e_{1}, h_{1}, \ldots, e_{k}, h_{k}}
\] 
from $v_{0}$ to $v_{0}$ (Property (\ref{BenignCoset}) of the definition of benign implies that we may compute a reduced representative for 
any non-reduced path).
If $k>0$, attach at $u_{0}$ a loop with $k$ edges having labels  
\[
(h_{0}, e_{1}, 1), (h_{1}, e_{2}, 1), \ldots, (h_{k-2}, e_{k-1}, 1), (h_{k-1}, e_{k}, h_{k}).
\]
These labels define the types of the vertices along this path, i.e. the first has type $o(e_{1})=v_{0}$, the second has type $o(e_{2})$, and so on.
Each vertex group along this path, except for $B_{u_{0}}$, is set to be trivial.  Set $B_{u_{0}}$ to be 
\[
B_{u_{0}} = \langle  h\in X \; | \; |p_{h}|=0 \rangle\leq A_{v_{0}}.
\]
Denote this $\A$-graph by $\Hzx$ and note that it is precisely the $\A$-graph $\mathcal{B}_{0}$ defined 
in Definition~5.1 of \cite{KWM05}.  

Now let $x$ be represented by the reduced $\A$-path 
\[
p_{x} = \Path{x_{0}, \ex_{1}, x_{1}, \ldots, \ex_{l}, x_{l}}
\]
from $v_{0}$ to $v_{0}$.  Construct a `line' of edges $\fx_{1}, \fx_{2},\ldots, \fx_{l}$, with $t(\fx_{i})=o(\fx_{i+1})$ for $i=1,\ldots,l-1$, 
having labels 
\[
(x_{0}, \ex_{1}, 1), (x_{1}, \ex_{2}, 1), \ldots, (x_{l-2}, \ex_{l-1}, 1), (x_{l-1}, \ex_{l}, x_{l}),
\]
with vertex types assigned according these labels as above. 
Attach this line to $\Hzx$ by setting $t(\fx_{l})=u_{0}$.  Denote $o(\fx_{1})=u_{x}$, the initial vertex of the first edge, 
and note that it is of type $v_{0}$.
Each vertex group along the path is set to be trivial (except $B_{u_{0}}$, which has been assigned above). 
We have now constructed $\Hzx$.

\begin{lemma}\label{Lem:StartGraph}
Let $\Hzx$ be as above.  Then 
\[
\overline{L(\Hzx, u_{0}, u_{0})} = H \;\mbox{  and  }\; \overline{L(\Hzx, u_{x}, u_{0})}=xH.
\]
\end{lemma}
\begin{proof}
Since the vertex groups along the line $\fx_{1},\ldots,\fx_{l}$ are all trivial, a reduced path from $u_{0}$ to $u_{0}$ cannot contain any of these edges. 
Hence the first statement is immediate (cf. Lemma~5.3 of \cite{KWM05}).

For the second statement, let $p$ be any reduced path from $u_{x}$ to $u_{0}$.  Since all the vertex groups along the $\fx_{1},\ldots,\fx_{l}$ branch are trivial, 
$p$ can be written as a concatenation of paths $p=qp'$ where  
\begin{equation}\label{Eqn:xpath}
q = \Path{1, \fx_{1}, 1, \fx_{2}, \ldots, 1, \fx_{l}, 1}
\end{equation}
and $p'$ is a reduced path from $u_{0}$ to $u_{0}$.  Clearly $\overline{\mu(q)}=x$, and $\overline{\mu(p')}=h\in H$ by the first statement, 
so $\overline{\mu(p)}=xh\in xH$, hence $\overline{L(\Hzx, u_{x}, u_{0})}\subset xH$.  Conversely, each $xh$ can be represented by a path 
$qp'$ as above and hence $\ xH\subset \overline{L(\Hzx, u_{x}, u_{0})}$.
\end{proof}

We now apply folding moves to $\Hzx$, in any order, producing a sequence of $\A$-graphs 
\begin{equation}\label{Eqn:FoldingSequence}
\Hnx{0}\cs \Hnx{1}\cs\ldots
\end{equation} 
We adopt the convention that the images in every $\Hnx{i}$ of edges $f$ and vertices $u$ from $\Hzx$ continue to be denoted 
by the same letters $f$, $u$. As well, $\mu$ will denote the map from $\Hnx{i}$-paths to $\A$-paths, for all values of $i$.

If at any point in the sequence the next move will identify $u_{x}$ with $u_{0}$, we stop and do not perform this move (the algorithm will 
output $y$ in this case, as we will see below).  We assume then that $u_{x}$ and $u_{0}$ are distinct in every $\Hnx{i}$.
We also make two small changes to the folding moves, regarding the use of auxiliary move A0.  
Recall from \cite{KWM05} that to apply A0 at a vertex $B_{u}$ we take an element $g\in A_{[u]}$ and replace 
\begin{enumerate}[(i)]
\item $B_{u}$ by $g B_{u}g^{-1}$,
\item $f_{\alpha}$ by $g f_{\alpha}$ for edges $f$ with $o(f)=u$, and 
\item $f_{\omega}$ by $f_{\omega}g^{-1}$ for edges $f$ with $t(f)=u$.
\end{enumerate}
Move A0 is only used in folding moves F1 and F2.

In F1, two distinct non-loop edges $f_{1}$ and $f_{2}$ with $o(f_{1})=o(f_{2})$ and $t(f_{1})\neq t(f_{2})$ are folded together.  
At the beginning of this move, A0 must be applied at one of the two vertices $t(f_{1})$ or $t(f_{2})$ (we may choose which) in order 
to obtain identical edge labels on $f_{1}$ and $f_{2}$.
Since no folding move in the sequence (\ref{Eqn:FoldingSequence}) identifies $u_{x}$ with $u_{0}$, at least one of the vertices 
$t(f_{1})$ or 
$t(f_{2})$ is neither $u_{x}$ nor 
$u_{0}$.  We choose to apply A0 at this vertex.
  
In F2, a non-loop edge $f_{2}$ is folded onto a loop edge $f_{1}$, where $o(f_{1})=o(f_{2})$.  Auxiliary move A0 must be applied at $t(f_{2})$, with an element 
$g\in A_{[t(f_{1})]}$, in order to equalize the edge labels on $f_{1}$ and $f_{2}$.  If $t(f_{2})$ is equal to $u_{x}$ or $u_{0}$, we will, 
at the conclusion of the folding move,  
apply A0 with the element $g^{-1}$ at the vertex $t(f_{2})$ (which now coincides with $t(f_{1})$).  Note that in the description of F2 in \cite{KWM05}, 
this is done only when $t(f_{2})=u_{0}$.  

It follows from \cite{KWM05} Prop.~4.15 that $\overline{L(\Hnx{i}, u_{0}, u_{0})}=H$ for all $i$.
We claim that 
\[
\overline{L(\Hnx{i}, u_{x}, u_{0})}=xH
\]
for every $i$.
From the description of the folding moves, and the considerations regarding A0 above, it follows that 
for every path $p$ in $\Hzx$ from $u_{x}$ to $u_{0}$ there exists, in each $\Hnx{i}$, a path $p'$ from 
$u_{x}$ to $u_{0}$ such that $\overline{\mu(p')}=\overline{\mu(p)}$.  Hence $\overline{L(\Hnx{i}, u_{x}, u_{0})}\supseteq xH$.
In particular, for the path $q$ which represents $x$ in $\Hzx$, there exists a path $q'$ in $\Hnx{i}$ from $u_{x}$ to $u_{0}$ such 
that $\overline{\mu(q')}=x$.
For the opposite inclusion, 
let $p$ be any path from $u_{x}$ to $u_{0}$ in $\Hnx{i}$.  Since $(q')^{-1}p$ is a loop based at $u_{0}$, 
we have $\overline{\mu((q')^{-1}p)}=h\in H$ and so 
\[
\overline{\mu(p)} = \overline{\mu(q' (q')^{-1} p)} = \overline{\mu(q')}\phantom{\cdot} \overline{\mu((q')^{-1}p)} = xh,
\]
proving the claim.

Now if at no point in the folding sequence is there a move that would identify $u_{x}$ with $u_{0}$, the folding algorithm 
terminates producing the folded graph $\Hx$ as required (see the proof of Theorem~5.8 of \cite{KWM05}).  Assume then 
that for the graph $\Hnx{i}$ there is an applicable folding move which would identify $u_{x}$ with $u_{0}$.

\begin{figure}
\begin{center}
\includegraphics[scale=0.8]{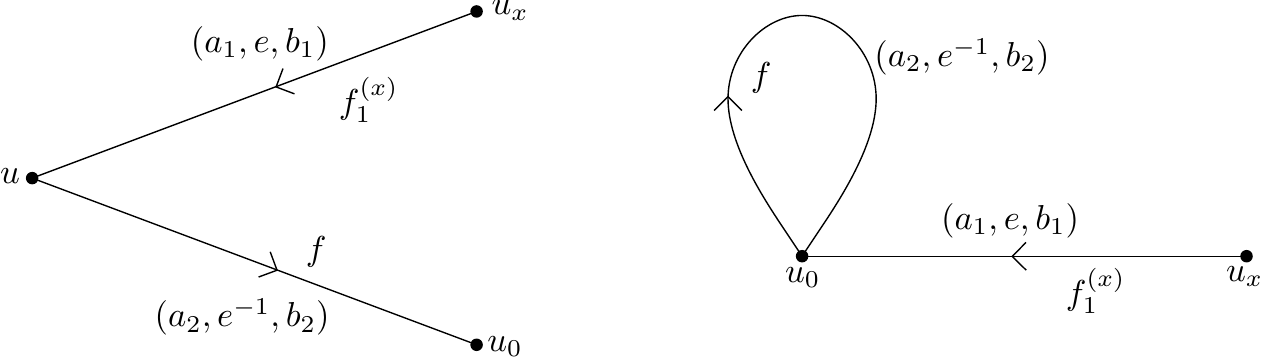}
\caption{Setup for an F1 fold (left) or an F2 fold (right) which would identify $u_{x}$ with $u_{0}$.}
\label{Fig:F1F2fold}
\end{center}
\end{figure}

This move must be of type F1 or F2, involving edges $(\fx_{1})^{-1}$ and an edge $f$ with 
$o(f)=o((\fx_{1})^{-1})$ and $t(f)=u_{0}$. Two possibilities 
are shown in Figure~\ref{Fig:F1F2fold} (the third possibility has $\fx_{1}$ as the loop edge).
Let $e=\ex_{1}$, $u=o(f)$, and let $\fx_{1}$ have label $(a_{1}, e, b_{1})$ and $f$ have label $(a_{2}, e^{-1}, b_{2})$.  Since 
$\Hnx{i}$ is not folded, there exist $b\in B_{u}$ and $c\in A_{e}$ such that 
\[
b_{1}b a_{2} = \omega_{e}(c).
\]
Then 
\[
p = \Path{1, \fx_{1}, b, f, 1}
\]
is a path in $\Hnx{i}$ from $u_{x}$ to $u_{0}$ hence $\overline{\mu(p)}=xh$ for some $h\in H$.  But 
\[
\mu(p) = \Path{a_{1}, e, b_{1} b a_{2}, e^{-1}, b_{2}} = \Path{a_{1}, e, \omega_{e}(c), e^{-1}, b_{2}} 
\]
which is equivalent to the length 0 path $y=a_{1} \alpha_{e}(c)b_{2}\in A_{u_{0}}$. Hence $xH=yH$, and the algorithm returns 
the element $y$.

\end{proof}

%==========================================================================================
%==========================  SEMI-CANONICAL FORMS =========================================
%==========================================================================================

\subsection{Semi-canonical forms for paths in certain $\A$-graphs}\label{Section:SemiCanonicalForms}

Let $\B$ be a folded $\A$-graph and let 
\[
p = \Path{p_{0}, e_{1}, p_{1}, e_{1}, \ldots, p_{n-1}, e_{n}, p_{n}}
\]
be an $\A$-path.  If there exists a $\B$-path $q$ such that $\overline{\mu(q)}=\overline{p}$, then there will (usually) be infinitely many 
equivalent such paths.  We will define a semi-canonical form for these paths, which will depend on the choice of representative $p$ for 
the class $\overline{p}$.  
While semi-canonical forms are not required to solve the problem of 
determining the existence of $q$, they are essential in solving the more difficult problem of whether or not there exists $q$ and $m\neq 0$ such that 
$\overline{\mu(q)}=\overline{p^{m}}$ (discussed in the next section).

We place the following restriction on $\B$:
assume that every edge group 
$B_{f}$ of $\B$ which is infinite has finite index in the corresponding edge group $A_{e}$ of $\A$.
This is satisfied, for example, when all edge groups of $\A$ are virtually cyclic.  
For each infinite $B_{f}$, fix a (finite) set $R_{f}$ of left coset representatives of $B_{f}$ in $A_{e}$.
%, using 
%$R_{f}=\{1\}$ in the case when $B_{f}=A_{f}$.

Our definition of semi-canonical forms corresponds to the following procedure for searching for a path 
$q=\Path{q_{0}, f, \ldots}$ such that $\overline{\mu(q)}=\overline{p}$.  We start searching at a vertex $\ustart$ such that 
$[\ustart]=o(e_{1})$. We locate an edge $f$ of type $e_{1}$ incident on $\ustart$.  
Since the map $\mu$ will multiply $q_{0}$ on the right by $f_{\alpha}$, and the target element $p_{0}$ may be replaced by 
$p_{0}\alpha_{e_{1}}(c)$ for any $c\in A_{e_{1}}$, the element  
$p_{0}\alpha_{e_{1}}(c)f_{\alpha}^{-1}$ must be in the vertex group 
$B_{\ustart}$ for some $c \in A_{e_{1}}$ for $q$ to exist.  
Having selected $q_{0}=p_{0}\alpha_{e_{1}}(c)f_{\alpha}^{-1}$, we proceed to the next vertex group, but 
with $p_{1}$ replaced by $\omega_{e_{1}}(c^{-1})p_{1}$.  Though there may be infinitely many choices for the `adjustment' $c$, we will show 
that only finitely many `canonical' choices need to be considered.  Along the next edge we may have another adjustment $c'$, and so produce 
a sequence of adjustments, each depending on the previous.

\subsubsection*{Canonical adjustments}

Let $f$ be an edge of $\B$ and denote $[f]=e$ and $u=o(f)$.  
For an element $c\in B_{f}$, we define the \emph{left adjustment} corresponding to $c$ 
as 
\begin{equation}\label{Eqn:LeftAdjustment}
l_{c} = \alpha_{e}(c)f_{\alpha}^{-1}
\end{equation}
and the \emph{right adjustment} corresponding to $c$ as 
\begin{equation}\label{Eqn:RightAdjustment}
r_{c} = f_{\omega}^{-1} \omega_{e}(c^{-1}).
\end{equation}
Note that while $l_{c}$ and $r_{c}$ depend on both $c$ and $f$, we will 
assume that the edge groups of $\A$ are pairwise disjoint, making $f$ uniquely determined by $c$.

For an element $a\in A_{o(e)}$ define a set $C(f,a)\subseteq A_{e}$ called the \emph{canonical adjustment set} as 
follows.  If there is no element $c_{0}\in A_{e}$ such that $al_{c_{0}}\in B_{u}$, then $C(f,a)$ is empty.  If such an element $c_{0}$ 
does exist, then 
\begin{enumerate}[(i)]
%\item if $A_{e}=1$, \label{Eqn:C1} %or $B_{f}=A_{e}$, 
%\[
%C(f,a) = \{ 1 \};
%\]
\item if $B_{f}$ is finite, \label{Eqn:C2}
\[
C(f,a) = c_{0}B_{f};%\{c_{0}b\;|\; b\in B_{f}\};
\]
\item if $B_{f}$ is infinite, %otherwise $B_{f}$ has finite index in $A_{e}$ and we set \label{Eqn:C3}
\[
C(f,a) = \{c \in R_{f}\;|\; al_{c}\in B_{u}\}.
\]

\end{enumerate}
In the case when $B_{f}$ is finite, the set $C(f,a)$ does not depend on the choice of $c_{0}$.  
Indeed, suppose we replace $c_{0}$ by another element $c_{0}'\in A_{e}$ such that 
$al_{c_{0}'}\in B_{u}$.  Then $(al_{c_{0}})^{-1}al_{c_{0}'}\in B_{u}$, and  
\[
(al_{c_{0}})^{-1}al_{c_{0}'} = 
\left(f_{\alpha}\alpha_{e}(c_{0}^{-1})a^{-1}\right)\left( a\alpha_{e}(c_{0}')f_{\alpha}^{-1}\right)
	= f_{\alpha}\alpha_{e}(c_{0}^{-1}c_{0}')f_{\alpha}^{-1},
\]
hence $\alpha_{e}(c_{0}^{-1}c_{0}')\in B_{u}^{f_{\alpha}}$.  Since $\B$ is 
folded, the edge group $B_{f}$ is equal to $\alpha_{e}^{-1}(B_{u}^{f_{\alpha}})$, hence 
$c_{0}^{-1}c_{0}'\in B_{f}$ because $\alpha_{e}$ is injective and the cosets coincide. 
%It follows that $C(f,a)$ does not depend on 
%the choice of $c_{0}$. 
\label{Page:Cfa}
%We will use the elements of $C(f,p)$
\begin{lemma}\label{Lem:Cfp}
Every element $c\in C(f,a)$ satisfies $al_{c}\in B_{u}$, and $C(f,a)$ is non-empty if and only if there exists $c\in A_{e}$ such 
that $al_{c}\in B_{u}$.
\end{lemma}
\begin{proof}
The first statement need only be verified in the case when $B_{f}$ is finite. Let $c_{0}b\in C(f,a)$ and 
recall that $\alpha_{f}:B_{f}\rightarrow B_{u}$ is defined by 
$\alpha_{f}(z) = f_{\alpha} \alpha_{e}(z) f_{\alpha}^{-1}$.  We have 
\[
al_{c_{0}b} =a\alpha_{e}(c_{0})\alpha_{e}(b)f_{\alpha}^{-1}=a\alpha_{e}(c_{0})f_{\alpha}^{-1}\alpha_{f}(b)=al_{c_{0}}\alpha_{f}(b)\in B_{u},
\]
as required.

For the second statement, we need only prove that if 
$B_{f}$ is infinite and there exists $c\in A_{e}$ such that $al_{c}\in B_{u}$ then 
there exists $c'\in R_{f}$ such that $al_{c'}\in B_{u}$.  Write $c=c'b$, where $c'\in R_{f}$ and $b\in B_{f}$. 
Since $\alpha_{f}(b^{-1})\in B_{u}$, we have
\[
al_{c'} = a\alpha_{e}(c')f_{\alpha}^{-1} \alpha_{f}(b)\alpha_{f}(b^{-1})= \left(a \alpha_{e}(c) f_{\alpha}^{-1}\right) \alpha_{f}(b^{-1}) 
	= al_{c} \alpha_{f}(b^{-1}) \in B_{u}
\]
as required.

%except in case (\ref{Eqn:C1}) when $B_{f}=A_{e}=\integers$.  
%In this case $-c\in B_{f}$, hence $\alpha_{f}(-c)\in B_{o(f)}$. 

%Then 
%\[
%pl_{0} = p\alpha_{e}(0)f_{\alpha}^{-1} = \left(p \alpha_{e}(c) f_{\alpha}^{-1}\right) \alpha_{f}(-c) \in B_{u},
%\]
%as required.

\end{proof}

\subsubsection*{Adjustment sequences}

Let 
\[
p=\Path{p_{0},e_{1},p_{1}, \ldots, e_{n}, p_{n}}
\]
be a reduced $\A$-path and let $\ustart, \uend$ be vertices of $\B$ with $[\ustart]=o(e_{1})$ and $[\uend]=t(e_{n})$.  
Let 
\[
\mathcal{F} = (f_{1},f_{2},\ldots,f_{n})
\]
be an edge path in $\B$ with $o(f_{1})=\ustart$, $t(f_{n})=\uend$, and $[f_{i}]=e_{i}$ for $i=1,\ldots,n$.  An 
\emph{adjustment sequence} associated with the pair $(p,\mathcal{F})$ is a sequence 
\[
\sigma = (c_{1}, c_{2}, \ldots, c_{n})
\]
with $c_{i}\in A_{e_{i}}$ for $i=1,\ldots,n$ 
that satisfies the following properties:
\begin{enumerate}[(i)]
\item $c_{1} \in C(f_{1}, p_{0})$, 
\item $c_{i} \in C(f_{i}, r_{c_{i-1}}p_{i})$ for $i=2,\ldots, n-1$, and 
\item $r_{c_{n}}p_{n} \in B_{t(f_{n})}$. 
\end{enumerate}
Note that the possible values for $c_{i}$ depend on $c_{i-1}$.
%adjustment sequences are constructed starting from the left, since $C(f_{i}, r_{c_{i-1}}p_{i})$ depends on $c_{i-1}$.
To every adjustment sequence $\sigma$ we associate a $\B$-path 
\[
Q(\sigma) = \Path{p_{0}l_{c_{1}},f_{1},r_{c_{1}}p_{1}l_{c_{2}},f_{2},\ldots,f_{n},r_{c_{n}}p_{n}}.
\]
If the path $\mathcal{F}$ has length 0, we associate with $\mathcal{F}$ an empty adjustment sequence $\sigma=\emptyset$ and define 
$Q(\sigma)=p_{0}$.  The path $Q(\sigma)$ is a $\B$-path since the elements 
$p_{0}l_{c_{1}}, r_{c_{1}}p_{1}l_{c_{2}}, \ldots, r_{c_{n}}p_{n}$ lie in the appropriate $\B$-vertex groups, by construction 
(see Lemma~\ref{Lem:Cfp}).

\begin{lemma}\label{Lem:SCPathIsCorrect}
For every adjustment sequence $\sigma$ associated with $(p,\mathcal{F})$, the path $Q(\sigma)$ is a $\B$-path satisfying
\[
\overline{\mu(Q(\sigma))} = \overline{p}.
\]
\end{lemma}
\begin{proof}
We have 
\begin{eqnarray*}
\overline{\mu(Q(\sigma))} & = & \overline{\Path{p_{0} l_{c_{1}}(f_{1})_{\alpha}, e_{1}, (f_{1})_{\omega}r_{c_{1}}p_{1}l_{c_{2}}(f_{2})_{\alpha}, 
	\ldots, e_{n}, (f_{n})_{\omega}r_{c_{n}} p_{n}}} \\
	& = & \overline{\Path{p_{0}\alpha_{e_{1}}(c_{1}), e_{1}, \omega_{e_{1}}(c_{1}^{-1})p_{1}\alpha_{e_{2}}(c_{2}),\ldots, e_{n},\omega_{e_{n}}(c_{n}^{-1})p_{n}}} \\
	& = & \overline{p}.
\end{eqnarray*}
\end{proof}

\subsubsection*{Semi-canonical paths}

Now let $\Phi$ be the set of all edge paths $\mathcal{F}=(f_{1},\ldots,f_{n})$ 
with $o(f_{1})=\ustart$, $t(f_{n})=\uend$, and $[f_{i}]=e_{i}$ for $i=1,\ldots,n$.
We define the set 
of \emph{semi-canonical paths} $\SC{p}{\ustart}{\uend}$ as 
\[
\SC{p}{\ustart}{\uend} = \bigcup_{\mathcal{F}\in \Phi}\{ Q(\sigma) \; | \; \mbox{$\sigma$ is an adjustment sequence for $(p,\mathcal{F})$}\}.
\] 
The key properties of semi-canonical paths (Theorem~\ref{Thm:SemiCanonicalForms}) are that every path mapping onto $\overline{p}$ is 
equivalent to a semi-canonical path and that the number of semi-canonical 
paths is finite.  
%For every index $0\leq i \leq n$, we define a prefix and a suffix path of $p$ by 
%the \emph{length $i$ prefix path} and \emph{length $i$ suffix path} are defined, respectively, by 
%\begin{eqnarray*}
%\head{p}{i} & = & \Path{p_{0}, e_{1}, p_{1}, \ldots, e_{i}, p_{i}} \\
%\tail{p}{i} & = & \Path{p_{i}, e_{i+1}, p_{i+1}, \ldots, e_{n}, p_{n}}.
%\end{eqnarray*}
%Note that the element $p_{i}$ is included in both paths.

\begin{theorem}\label{Thm:SemiCanonicalForms}
Let $\A$ be a graph of groups and let $\B$ be a folded $\A$-graph such that 
every infinite edge group of $\B$ has finite index in the corresponding edge group of $\A$.  
Let 
\[
p=\Path{p_{0}, e_{1}, p_{1},\ldots, e_{n},p_{n}}
\]
be an $\A$-path and let $\ustart$, $\uend$ be vertices of $\B$ with $[\ustart]=o(e_{1})$ and $[\uend]=t(e_{n})$.
Then every $\B$-path $q$ from $\ustart$ to $\uend$ that satisfies 
\begin{equation}\label{Eqn:qmapstop}
\overline{\mu(q)}=\overline{p}
\end{equation}
is equivalent to a path in the finite set $\SC{p}{\ustart}{\uend}$.  Consequently,  
$\SC{p}{\ustart}{\uend}$ is non-empty if and only if there exists a path $q$ from $\ustart$ to $\uend$ satisfying (\ref{Eqn:qmapstop}).

%\begin{enumerate}[(i)]
%\item the set $\SC{p}{\ustart}{\uend}$ is finite; 
%\item for every $\B$-path $q$ from $\ustart$ to $\uend$ such that $\overline{\mu(q)}=\overline{p}$, 
%there exists an equivalent path $q'\in \SC{p}{\ustart}{\uend}$;\label{existsEquivalent}
%\item there exists a $\B$-path $q$ from $\ustart$ to $\uend$ such that $\overline{\mu(q)}=\overline{p}$ if and only if 
%$\SC{p}{\ustart}{\uend}$ in non-empty.\label{SCnonemptyiffpath}
%\end{enumerate}

\end{theorem}
\begin{proof}
The fact that $\SC{p}{\ustart}{\uend}$ is finite follows from the fact that every set $C(f,a)$ is finite and there are finitely many edge paths 
of length $n$ in $\B$, hence there are finitely many adjustment sequences.  

We now proceed by induction on the length $n$ of $p$.  
Suppose that the path $q$ exists.  We may assume that $q$ is reduced.  
When $n=0$, $q$ must be precisely the length 0 path $q=\Path{p_{0}}$, and $\ustart=\uend$.  This 
path is the unique element of $\SC{p}{\ustart}{\uend}$.

Now assume the statement holds for paths of length less than $n$ and 
suppose that there exists a path 
$q=\Path{q_{0},f_{1},q_{1},\ldots,f_{n}, q_{n}}$ from $\ustart$ to $\uend$ such that $\overline{\mu(q)}=\overline{p}$.
Denote 
\[
\tail{q}{i}  =  \Path{q_{1}, f_{2}, q_{3}, \ldots, f_{n}, q_{n}}.
\]
Then there exists $c_{1}\in A_{e_{1}}$ such that $q_{0}(f_{1})_{\alpha}=p_{0}\alpha_{e_{1}}(c_{1})$ 
and $\tail{q}{1}$ satisfies
\[
\overline{\mu(\tail{q}{1})} = \overline{r_{c_{1}}\tail{p}{1}}.
\]
%Note that $q_{0}=p_{0}l_{c_{1}}\in B_{\ustart}$.  
Let $f=f_{1}$.  
We will show that we may assume $c_{1}\in C(f,p_{0})$.
%by replacing $q$ with 
%an equivalent path $\tilde{q}$. %We consider the three cases in the definition of $C(f,p_{0})$. 

If $B_{f}$ is finite, then $C(f,p_{0})=c_{0}B_{f}=c_{1}B_{f}$ (see the proof that $C(f,a)$ does not depend on $c_{0}$, page \pageref{Page:Cfa}),
so $c_{1}\in C(f,p_{0})$ already.
If $B_{f}$ is infinite, there exist $c_{1}'\in R_{f}$ and $b\in B_{f}$ 
such that $c_{1}=c_{1}'b$.  Since $b^{-1}\in B_{f}$, we may replace $q$ by the equivalent path 
\[
\tilde{q}=\Path{q_{0}\alpha_{f}(b^{-1}), f, \omega_{f}(b)\tail{q}{1}}.
\]
Since the first element $q_{0}\alpha_{f}(b^{-1})$ of $\tilde{q}$ satisfies
\[
\left( q_{0}\alpha_{f}(b^{-1})\right) f_{\alpha} = q_{0} f_{\alpha} \alpha_{e_{1}}(b^{-1})
	= p_{0}\alpha_{e_{1}}(c_{1}b^{-1}) = p_{0}\alpha_{e_{1}}(c_{1}'), 
\]
we may assume from the beginning that $c_{1}\in C(f,p_{0})$. 

Since $\tail{q}{1}$ has length $n-1$ it is equivalent, by induction, to a semi-canonical path.  That is, there exists 
an adjustment sequence $\sigma'=(c_{2},\ldots,c_{n})$ associated 
with the path $r_{c_{1}}\tail{p}{1}$ and the edge path $(f_{2},\ldots,f_{n})$ 
such that $\tail{q}{1}$ is equivalent to $Q(\sigma')$. 
If $n=1$ then $r_{c_{1}}p_{1}\in B_{\uend}$ and if $n\geq 2$ then $c_{2}\in C(f_{2},r_{c_{1}}p_{1})$, 
hence the sequence  
$\sigma=(c_{1},\ldots,c_{n})$ is an adjustment sequence associated with $p$ and $(f_{1},\ldots,f_{n})$. Then 
\[
\overline{q}  =  \overline{\Path{q_{0},f_{1},Q(\sigma')}}
	= \overline{\Path{p_{0}l_{c_{1}}, f_{1}, r_{c_{1}}p_{1}l_{c_{2}},f_{2},\ldots,r_{c_{n}}p_{n}}} = \overline{Q(\sigma)}
\]
hence $q$ is equivalent to the semi-canonical path $Q(\sigma)$.

\end{proof}

\begin{remark}
If $p'$ is an $\A$-path with $\overline{p'}=\overline{p}$, the set $\SC{p'}{\ustart}{\uend}$ need not coincide with $\SC{p}{\ustart}{\uend}$, 
hence the adjective `semi-canonical'.  This arises from the fact that whenever $B_{f_{i}}$ is finite, the set $C(f_{i},r_{c_{i-1}}p_{i})$ depends 
on $p_{i}$.
\end{remark}

%==========================================================================================
%=================== READING POWERS IN AN A-GRAPH ========================================
%=========================================================================================

\subsection{Reading powers in an $\A$-graph}\label{Section:ReadingPowers}

When $\A$ is a benign graph of groups and $\B$ is a folded $\A$-graph, there is an algorithm that, 
given an $\A$-path $p$, decides 
whether or 
not there exists a $\B$-path $q$ such that $\overline{\mu(q)}=\overline{p}$ (see
Claim~5.14 of \cite{KWM05}).  We require a stronger version of this result (the algorithm \textsc{ReadPower} on page \pageref{Alg:ReadPower}),
which will decide whether or not there exists $q$ such that $\overline{\mu(q)}=\overline{p^{m}}$ for some $m>0$, under certain 
conditions on $\A$ and $\B$. 

We use Theorem~\ref{Thm:SemiCanonicalForms} to restrict the search for $q$ to semi-canonical forms, but consequently we must  
insist that infinite edge groups of $\B$ have finite index 
in the corresponding $\A$-edge groups.  
%In our primary application, $\A$ has a single edge group $\integers$ so this 
%condition is satisfied.  
We assume that $\B$ is already folded, so rather than insist that $\A$ be benign (the requirement for folding), 
we specify a set of properties (Property~\ref{Property:PowerReadingConditions}) for the pair $(\A, \B)$ 
that are sufficient for our algorithm to run and terminate. To simplify the description of these properties, we begin 
by defining the following property.
%These properties are presented in more generality than 
%is required for \S\ref{Section:EffectiveCoherence}.

%We will need the following property to hold for certain vertex groups $G$ and adjacent edge groups 
%$C$, $C'$, as described in Property~\ref{Property:PowerReadingConditions}.  It is used to guarantee termination of the 
%\textsc{ReadPower} algorithm.
%and severely limits the situations in which \textsc{ReadPower} may be used. 

\begin{property}\label{Property:DoubleCoset}
Let $G$ be a group, and let $C$ and $C'$ be subgroups of $G$. 
\begin{enumerate}[(i)]
\item We say that $G$ satisfies \emph{Property~\ref{Property:DoubleCoset}} with respect to $C$ if for every finitely generated subgroup 
$H\leq G$ and for every $x,g\in G$ with $gx^{-1}\not\in C$, if 
\[
\mbox{$H\cap C$ and $H\cap C^{x}$ are both finite}
\] 
then there exist finitely many pairs $(c,c')\in C\times C^{x}$ such that 
\[
cgc'\in H.
\]
\item We say that $G$ satisfies \emph{Property~\ref{Property:DoubleCoset}} with respect to $(C,C')$ if for every finitely generated subgroup 
$H\leq G$ and every $g\in G$, if 
\[
\mbox{$H\cap C$ and $H\cap C'$ are both finite}, 
\]
then there exist finitely many pairs $(c,c')\in C\times C'$ such that 
\[
cgc'\in H.
\]
\end{enumerate}
\end{property}

We now state the conditions we will need in order to use \textsc{ReadPower}. 
By an \emph{edge cycle} we mean a finite edge path $(f_{1},f_{2},\ldots,f_{m})$ with $t(f_{m})=o(f_{1})$.
Let $f_{m+1}$ denote $f_{1}$.

\begin{property}\label{Property:PowerReadingConditions}
Let $\A$ be a graph of groups and let $\B$ be a folded $\A$-graph.  
We say that the pair $(\A, \B)$ satisfies \emph{Property~\ref{Property:PowerReadingConditions}} if all of the following conditions hold.
\begin{enumerate}[(i)]
\item \label{PowerReading:VertexMembership} Every vertex group of $\A$ has decidable membership and power membership problem.
\item There is an algorithm that, given a vertex $v$ of $\A$, an edge $f$ of $\B$ with $[o(f)]=v$, and an element $x\in A_{v}$, 
decides whether or not the intersection 
\[
xB_{o(f)}\cap \alpha_{[f]}(A_{[f]})^{f_{\alpha}^{-1}}
\]
is empty. \label{PowerReading:Intersection}
%\item \label{PowerReading:Folded} $\B$ is folded.
\item \label{PowerReading:FiniteIndex} For every edge $f$ of $\B$, if $B_{f}$ is infinite then it has finite index in $A_{[f]}$.  
It must be known which $B_{f}$ are infinite.
\item \label{PowerReading:DoubleCoset} 
	For every edge cycle $f_{1},\ldots, f_{m}$ of $\B$ in which $B_{f_{i}}$ is finite for all $i$, there exists $i\in \{1,\ldots,m\}$ such that 
	\begin{enumerate}[(1)]
	\item \label{PowerReading:DoubleCosets1} 
	if $e_{i+1}\neq e_{i}^{-1}$, then $A_{[t(f_{i})]}$ satisfies Property~\ref{Property:DoubleCoset} with respect to the pair of 
	subgroups 
	\[
	\left(\omega_{e_{i}}(A_{e_{i}})^{(f_{i})_{\omega}}, \alpha_{e_{i+1}}(A_{e_{i+1}})^{(f_{i+1})_{\alpha}^{-1}}\right),
	\]
	%then for every $x\in A_{[t(f_{i})]}$ there exist finitely many pairs $(c,c')\in B_{f_{i}}\times B_{f_{i+1}}$ such that 
	%\[
	%\omega_{f_{i}}(c) x \alpha_{f_{i+1}}(c')\in B_{t(f_{i})},
	%\]
	\item \label{PowerReading:DoubleCosets2} 
	if $e_{i+1}=e_{i}^{-1}$, then $A_{[t(f_{i})]}$ satisfies Property~\ref{Property:DoubleCoset} with respect to 
	\[
	\omega_{e_{i}}(A_{e_{i}})^{(f_{i})_{\omega}},
	\]
	%for every $x\in A_{[t(f_{i})]}\setminus \omega_{f_{i}}(B_{f_{i}})$ 
	%there exist finitely many pairs $(c,c')\in B_{f_{i}}\times B_{f_{i+1}}$
	%such that 
	%\[
	%\omega_{f_{i}}(c) x \alpha_{f_{i+1}}(c')\in B_{t(f_{i})}.
	%\]
	\end{enumerate}
	where $e_{i}=[f_{i}]$ and $e_{i+1}=[f_{i+1}]$.
\end{enumerate}
\end{property}

Let $p=\Path{p_{0}, e_{1}, p_{1}, \ldots, e_{n}, p_{n}}$ be a reduced $\A$-path which is a loop (i.e. 
$o(e_{1})=t(e_{n})$).  We say that 
$p$ is \emph{cyclically reduced} if $p^{2}$ is also reduced. 
For $i\in\{0,\ldots,n-1\}$, define 
\begin{eqnarray*}
\hat{p}_{0} & = & p_{n}p_{0}, \\
\hat{p}_{i}& = & p_{i}, \;\mbox{for $i\neq 0$}.
\end{eqnarray*}  
Indices of $\hat{p}_{i}$ are taken modulo $n$.  Note that for any $m>0$, 
\begin{equation}\label{Eqn:pn}
p^{m} = \Path{p_{0}, e_{1}, \hat{p}_{1}, e_{2}, \hat{p}_{2}, \ldots, e_{n}, \hat{p}_{0}, e_{1},\hat{p}_{1}, \ldots, e_{n},p_{n}}.
\end{equation}

The following algorithm \textsc{ReadPower} is a modified breadth-first search on the graph $\B$.  It uses a data structure  
$\mathcal{T}$ called the \emph{search tree} which is a rooted, directed, ordered tree.  Each vertex $\tau$ of $\mathcal{T}$ is labelled by 
a pair $(u,i)$ where $u$ is a vertex of $\B$ and $i$ is the distance modulo $n$ from the root to $\tau$, 
%$i\in \integers/n\integers=\{0,\ldots,n-1\}$, 
and is  
marked as either `explored' or `unexplored'.  Though $i$ is determined by the depth of $\tau$ in $\mathcal{T}$, having $i$ 
explicitly recorded will simplify 
our description and proof of the algorithm.
The root is labelled by $(\ustart, 0)$.
Each edge from a vertex labelled by $(u,i)$ to a vertex labelled by $(u^{*},i+1)$ 
is labelled by a pair $(f,c)$ where $f$ is an edge of $\B$ from $u$ to $u^{*}$ 
and $c\in A_{[f]}$.  The element $c$ will be a canonical adjustment which is `pushed forward' to the 
next vertex group.

%To each vertex $\tau$ of $\mathcal{T}$ with label $(u,i)$ we associate a $\B$-path $P(\tau)$.
%semi-canonical path $Q(\tau)$ in $\piB$ as follows.  
%Let $d$ be the distance from the root to $\tau$,  
%write $d=nm+i$, and 
%let $(f_{1},c_{1}),\ldots,(f_{d},c_{d})$ be the edge labels of the unique path from the root to $\tau$.  
%Then $Q(\tau)$ is the path 
%corresponding to the $\A$-path $\head{p^{m+2}}{d}$, the edge sequence 
%$f_{1},\ldots,f_{d}$, and the adjustment sequence $c_{1},\ldots,c_{d}$.  That is, 
%\begin{equation}\label{Eqn:q}
%\[
%P(\tau) = \Path{p_{0} l_{c_{1}}, f_{1}, r_{c_{1}}\hat{p_{1}}l_{c_{2}}, f_{2}, \ldots, f_{d}, r_{c_{d}} \hat{p_{d}}},
%\]
%\end{equation}
%where $l_{c_{i}}=\alpha_{e_{i}}(c_{i})(f_{i})_{\alpha}^{-1}$ and 
%$r_{c_{i}}=(f_{i})_{\omega}^{-1}\omega_{e_{i}}(-c_{i})$.

\begin{algorithm}\label{Alg:ReadPower}
\caption{\textsc{ReadPower}}
\KwIn{Graph of groups $\A$ and a folded $\A$-graph $\B$ satisfying Property~\ref{Property:PowerReadingConditions}, 
a reduced $\A$-loop $p=<p_{0},e_{1},\ldots,e_{n},p_{n}>$, and vertices $\ustart,\uend$ of $\B$.}
\KwOut{A pair $(q,m)$ where $q$ is a $\B$-path from $\ustart$ to $\uend$ and $m$ is a positive integer $m$ such that 
$\overline{\mu(q)}=\overline{p^{m}}$, with $m$ minimal over all such pairs.  If no pair exists, the word `No'.}
$\mathcal{T}:=\mbox{single root vertex labelled by $(\ustart,0)$ and marked `unexplored'}$\;
\If{$[\ustart]\neq o(e_{1})$ or $[\uend]\neq t(e_{n})$}{\Return `No'\;}
\While{$\mathcal{T}$ has unexplored vertices \label{Code:While}}{
	$\tau:=$ unexplored vertex having minimum distance \label{Code:Tau}
	$d$ from the root and, of all unexplored vertices at distance $d$, is left-most\;
	$(u,i):=\mbox{label of $\tau$}$\;
	$e:=e_{i+1}$\;
	\If{$\tau$ is the root}{
		\If{$\mathrm{Length}(p)=0$}{
			\If{$B_{\ustart}\cap\langle p_{0}\rangle \neq 1$\label{Code:BaseCaseIntersection}}{
				\Return $(p_{0}^{m},m)$ with $m>0$ minimal such that $p_{0}^{m}\in B_{\ustart}$\label{Code:BaseCaseReturn}\;
			}
			\Else{\Return `No'\;}
		}
		\Else{$a:=p_{0}$\;} \label{Code:a}
		}
	\Else{
		$\tau':=\mathrm{Parent}(\tau)$\;
		$(f',c'):=\mbox{label of edge $\tau'\rightarrow\tau$}$\;
		$a:=r_{c'}\hat{p}_{i}$\;%(f_{\omega}')^{-1}\omega_{e_{i}}(-c')\hat{p}_{i}$\;
	}
	$F := \{f\in \mathrm{Edges}(\B)\; | \; [f]=e,\, o(f)=u\}$\;
	\For{each $f\in F$}{
		\If{$\exists\, c\in A_{e}$ such that $al_{c}\in B_{u}$ \label{Code:Cf}}{
			\For{each $c\in C(f,a)$\label{Code:Cfstar}}{
				%$a':=r_{c}$\;%f_{\omega}^{-1}\omega_{e}(-c)$\;
				\If{$i=n-1$ and $t(f)=\uend$ and $r_{c}p_{n}\in B_{\uend}$ \label{Code:Termination}}{
					$m:=(d+1)/n$\;
					Let $(f_{1},c_{1}),\ldots,(f_{d},c_{d})$ be the edge labels of the path from the root to $\tau$\;
					$\sigma:=(c_{1},\ldots,c_{d},c)$\;
					\Return $(Q(\sigma), m)$ \label{Code:ReturnYes}\;
					}
				\Else{
					%$p':=a'\hat{p}_{i+1}$\;
					\If{no vertex $(t(f), i+1)$ in $\mathcal{T}$ has incoming edge $(f,c)$\label{Code:NewNode}}{
						Insert a child $\tau''$ of $\tau$ with label $(t(f),i+1)$ \label{Code:NewNode2}\;
						Label the edge $\tau\rightarrow\tau''$ by $(f,c)$\label{Code:NewNode3}\;
						Mark $\tau''$ as unexplored\;
						}
					}
				}
			}
		}
	Mark $\tau$ as explored\;
	}
	\Return `No'\label{Code:ReturnNo}\;

\end{algorithm}

\begin{theorem}\label{Thm:ReadingPowers}
There is an algorithm \textsc{ReadPower} that, given
\begin{itemize}
\item a graph of groups $\A$ and a folded $\A$-graph $\B$ satisfying the conditions of Property~\ref{Property:PowerReadingConditions}, 
\item a cyclically reduced $\A$-loop $p$, and
\item vertices $\ustart,\uend$ of $\B$,
\end{itemize}
decides whether 
or not there exists a $\B$-path $q$ from $\ustart$ to $\uend$ and an integer $m>0$ such that 
\[
\overline{\mu(q)}=\overline{p^{m}}
\]
 and if so, produces $q$ and the 
minimum such $m$.
\end{theorem}
\begin{proof}

% EFFECTIVENESS

\emph{Algorithm effectiveness.}
First, we check that all steps of the algorithm may in fact be carried out. 
The only non-obvious steps are those on lines 
\ref{Code:BaseCaseIntersection}, 
\ref{Code:BaseCaseReturn}, \ref{Code:Cf}, \ref{Code:Cfstar}, and \ref{Code:Termination}. 

Line \ref{Code:BaseCaseIntersection} is an instance of the power membership problem in $A_{[\ustart]}$, hence is decidable.
For line \ref{Code:BaseCaseReturn}, we check, using decidability of the membership 
problem in $A_{[\ustart]}$, if each of $p,p^{2},p^{3},\ldots$ is in $B_{\ustart}$ and return the first success.

On line \ref{Code:Cf}, deciding the existence of $c$ is equivalent to deciding if the intersection 
\[
\left(f_{\alpha}a^{-1}\right) B_{u}\cap \alpha_{e}(A_{e})^{f_{\alpha}^{-1}},
\] 
in the vertex group $A_{[u]}$, is non-empty.  This is decidable by Property~\ref{Property:PowerReadingConditions}(\ref{PowerReading:Intersection}).

On line \ref{Code:Cfstar}, we compute the set $C(f,a)$.  It is known which $B_{f}$ are infinite.
If $B_{f}$ is finite, we simply search for 
$c\in A_{e}$ such that $al_{c}\in B_{u}$, which is known to exist.  If $B_{f}$ is infinite, we first construct a set of coset representatives 
$R_{f}$ of $B_{f}$ in $A_{[f]}$ (we can construct the Schrier graph since membership is decidable). Then we check 
the condition $al_{c}\in B_{u}$ for each $c\in R_{f}$.  %, which are all instances of the membership problem in $A_{[u]}$.  
Line \ref{Code:Termination} also involves a membership problem 
in a vertex group of $\A$, and membership is decidable by Property~\ref{Property:PowerReadingConditions}(\ref{PowerReading:VertexMembership}).

% TERMINATION 

\emph{Algorithm termination.}
Next, we ensure that the algorithm terminates.  Every iteration of the \textbf{while} loop on line \ref{Code:While} marks one vertex of $\mathcal{T}$ 
as explored, so it suffices to show that the size of $\mathcal{T}$ is bounded.  Since $F$ and $C(f,a)$ are always finite sets, every vertex in 
$\mathcal{T}$ has finitely many children, so its suffices to show that the depth of $\mathcal{T}$ is bounded.

Observe that an 
edge of $\mathcal{T}$ labelled by $(f, c)$ having terminal vertex labelled by $(t(f), i)$ 
occurs at most once in $\mathcal{T}$: once such an edge exists, 
the conditional on line \ref{Code:NewNode} prevents it 
from being created a second time.  The number of possible values for $f$ and $i$ is finite.  If $B_{f}$ is infinite, then 
the set $C(f,a)$ is, regardless of $a$, a subset of the fixed, finite set of coset representatives $R_{f}$. It follows that 
there are finitely many edges in $\mathcal{T}$ labelled by $(f,c)$ such that $B_{f}$ is infinite.  
Note that when $B_{f}$ is finite the set 
$C(f,a)$ is finite but depends on $a$.  

Assume that the depth of $\mathcal{T}$ is unbounded.  Then $\mathcal{T}$ contains an infinite path  
\begin{equation}\label{Eqn:InfinitePath}
\tau_{0}\rightarrow\tau_{1}\rightarrow\tau_{2}\rightarrow\ldots.
\end{equation}
We denote by $(f_{j}, c_{j})$ the label of the edge from $\tau_{j}$ to $\tau_{j+1}$.
By the above observation, there exists $M>0$ such that 
$B_{f_{j}}$ is finite for all $j>M$. Let $N$ be the number of edges in $\B$.  For all $j>M$, the sequence of edges 
\[
f_{j},f_{j+1},\ldots,f_{j+N}
\]
must contain a subsequence which is a edge cycle in $\B$, and so must contain a sequential pair of edges $f,f'$ satisfying the statements 
in Property~\ref{Property:PowerReadingConditions}(\ref{PowerReading:DoubleCoset}).  Since the path (\ref{Eqn:InfinitePath}) is infinite, 
there exists such a pair that occurs infinitely often.  Hence there exists  
a sequential pair of edges $f,f'$ satisfying the statements in Property~\ref{Property:PowerReadingConditions}(\ref{PowerReading:DoubleCoset}), 
an integer $i\in \{0,\ldots,n-1\}$,
and an infinite subset $J\subset \naturals_{>M}$ such that for all $j\in J$, $f_{j-1}=f$, $f_{j}=f'$, and $\tau_{j}$ is labelled by $(u, i)$, 
where $u=t(f)$.

Let $e=[f]$, $e'=[f']$, and denote $C=\omega_{e}(A_{e})^{f_{\omega}}$ and 
$C'=\alpha_{e'}(A_{e'})^{(f')_{\alpha}^{-1}}$.  Since $\B$ is folded, the edge group $B_{f}$ is defined by 
\[
B_{f} = \omega_{e}^{-1}\left(B_{u}^{f_{\omega}^{-1}}\cap \omega_{e}(A_{e}) \right).
\]
Since $B_{f}$ is finite,  $B_{u}^{f_{\omega}^{-1}}\cap \omega_{e}(A_{e})$ is also finite hence 
$B_{u}\cap C$ is finite.
Since $B_{f}$ is also defined by 
\[
B_{f} = \alpha_{e'}^{-1}\left(B_{u}^{f'_{\alpha}}\cap \alpha_{e'}(A_{e'}) \right),
\]
it follows in the same manner that $B_{u}\cap C'$ is finite.

Consider first the case when $e'\neq e^{-1}$.  We know that $A_{[u]}$ satisfies Property~\ref{Property:DoubleCoset} with 
respect to $(C,C')$.
By construction, $r_{c_{j-1}}\hat{p}_{i}l_{c_{j}}\in B_{u}$ for all $j\in J$, and we may rewrite 
this element as 
%\begin{equation}\label{Eqn:Element}
\[
	\left(f_{\omega}^{-1} \omega_{e}(c_{j-1}^{-1})\right) \hat{p}_{i} \left(\alpha_{e'}(c_{j})(f')_{\alpha}^{-1}\right) =
	\omega_{e}(c_{j-1}^{-1})^{f_{\omega}} \left(f_{\omega}^{-1}\hat{p}_{i}(f')_{\alpha}^{-1}\right)\alpha_{e'}(c_{j})^{(f')_{\alpha}^{-1}}.
\]
%\end{equation}
Set $H=B_{u}$ and $g=f_{\omega}^{-1}\hat{p}_{i}(f')_{\alpha}^{-1}\in A_{[u]}$. Since $B_{u}\cap C$ and $B_{u}\cap C'$ are both finite,
%(\ref{Eqn:TrivialIntersection1}) and 
%(\ref{Eqn:TrivialIntersection2}) hold,
Property~\ref{Property:DoubleCoset} states 
that there are finitely many pairs $(c_{j-1}, c_{j})$ for which 
$r_{c_{j-1}}\hat{p}_{i}l_{c_{j}}\in B_{u}$.  Hence there exists $c$ such that $c_{j-1}=c$ for infinitely many $j\in J$.
This contradicts the fact that $\mathcal{T}$ contains at most one edge labelled by $(f, c)$ with terminal vertex labelled 
by $(t(f), i)$.
 
If $e'=e^{-1}$, we have $\alpha_{e'}=\omega_{e}$ and we may rewrite the element $r_{c_{j-1}}\hat{p}_{i}l_{c_{j}}$ as 
\[
	\omega_{e}(c_{j-1}^{-1})^{f_{\omega}} 
	\left(f_{\omega}^{-1}\hat{p}_{i}(f')_{\alpha}^{-1}\right)
	\left(\omega_{e}(c_{j})^{f_{\omega}}\right)^{f_{\omega}^{-1}(f')_{\alpha}^{-1}}.
\]
Set $H=B_{u}$, $g=\left(f_{\omega}^{-1}\hat{p}_{i}(f')_{\alpha}^{-1}\right)$, and $x=f_{\omega}^{-1}(f')_{\alpha}^{-1}$.  
Since $gx^{-1} = \hat{p_{i}}^{f_{\omega}}$, if $gx^{-1}\in C$ then $\hat{p_{i}}\in\omega_{e}(A_{e})$ which contradicts the 
fact that $p$ is cyclically reduced. Hence $gx^{-1}\not\in C$, so Property~\ref{Property:DoubleCoset} states that there are finitely 
many pairs $(c_{j-1}, c_{j})$ such that $r_{c_{j-1}}\hat{p}_{i}l_{c_{j}}\in B_{u}$ and we obtain a contradiction as above.

%then we may argue as above, applying Property~\ref{Property:PowerReadingConditions}(\ref{PowerReading:DoubleCosets2}) to obtain a contradiction, unless the element 
%$f_{\omega}^{-1}\hat{p}_{i}(f')_{\alpha}^{-1}=f_{\omega}^{-1}\hat{p}_{i}f_{\alpha}$ lies 
%in $\omega_{e}(A_{e})^{f_{\omega}}$.  If $f_{\omega}^{-1}\hat{p}_{i}f_{\alpha}=\omega_{e}(c)^{f_{\omega}}$ for some $c\in A_{e}$, then  
%\[
%f_{\omega} \left(f_{\omega}^{-1}\hat{p}_{i}f_{\alpha}\right) (f^{-1})_{\alpha}  = \omega_{e}(c)
%\]
%which contradicts the fact that $\B$ is folded.

% CORRECTNESS

\emph{Algorithm correctness.} Finally, we prove that the algorithm is correct.  The cases when $p$ has length zero and 
when $[\ustart]\neq o(e_{1})$ or $[\uend]\neq t(e_{n})$ are clearly correct. 

First, suppose the algorithm returns the pair $(Q(\sigma), m)$ on line \ref{Code:ReturnYes}. 
It follows immediately from lines \ref{Code:Cfstar} and \ref{Code:Termination} that 
$(c_{1},\ldots,c_{d},c)$ is an adjustment sequence for $p^{m}$, hence the desired path $q=Q(\sigma)$ exists. 
We will argue the minimality of $m$ below.

Now suppose that the algorithm reaches line \ref{Code:ReturnNo}, returning `No'.  Assume, for contradiction, that there exists a $\B$-path 
$q$ from $\ustart$ to $\uend$ and an integer $m>0$ such that $\overline{\mu(q)}=\overline{p^{m}}$, and assume that $m$ is minimal.  Then the set 
$\SC{p^{m}}{\ustart}{\uend}$ is non-empty, so there exist an edge sequence $\mathcal{F}=(f_{1},\ldots,f_{nm})$ 
and an adjustment sequence $\sigma=(c_{1},\ldots,c_{nm})$ associated with $(p^{m}, \mathcal{F})$.  
Since the algorithm searches exhaustively for adjustment sequences, it will discover this sequence, recording it in $\mathcal{T}$ during lines \ref{Code:NewNode2} and \ref{Code:NewNode3}, 
and reaching the return statement on line \ref{Code:ReturnYes}, 
unless the conditional on line \ref{Code:NewNode} fails.  Assume then that such a failure first occurs at position $j$ in the adjustment sequence: that is, 
\[
j=i \pmod{n},
\]
and 
the edge $\tau'\rightarrow\tau$ has label $(f_{j}, c_{j})$, but the required child $\tau''$ of $\tau$ is \emph{not} created 
because $\mathcal{T}$ already contains a vertex with label 
\[
(t(f_{j+1}), j+1)
\]
having incoming edge with label 
\[
(f_{j+1}, c_{j+1}).
\]
Let $\tau_{k+1}$ be this pre-existing vertex, let $\tau_{0},\tau_{1},\ldots,\tau_{k+1}$ be the vertices along the unique 
path from the root $\tau_{0}$ to $\tau_{k+1}$, and let $(f_{0}',c_{0}'),\ldots, (f_{k}',c_{k}')$ be the sequence of edge labels along 
this path. Note that $\tau_{k}$ is at distance $k$ from the root, the current search vertex $\tau$ (line \ref{Code:Tau}) is at 
distance $j$ from the root, and that $k=j\pmod{n}$.

Since $\tau_{k}$ has a child (and is not equal to $\tau$), it must be marked `explored'.  
Since vertices of $\mathcal{T}$ are processed in order of distance from the root then left-to-right, 
either $k<j$ or $k=j$ and $\tau_{k}$ appears to the left of $\tau$. 
Consider the sequence of elements 
\[
\sigma''=(c_{1}',c_{2}',\ldots,c_{k}',c_{j+1},c_{j+2},\ldots,c_{mn})
\]
and the edge path 
\[
\mathcal{F}''=(f_{1}',f_{2}',\ldots,f_{k}',f_{j+1},f_{j+2},\ldots,f_{mn}).
\]
Since 
\[
(f_{k+1}', c_{k+1}') = (f_{j+1}, c_{j+1}),
\]
it follows that $\sigma''$ is an adjustment sequence associated with either $(p^{m}, \mathcal{F''})$,  
in the case $k=j$, or with $(p^{m'}, \mathcal{F''})$ where 
\[
m' = \frac{j-k}{n},
\]
in the case $k<j$.  In the latter case, the path $Q(\sigma'')$ is in $\SC{p^{m'}}{\ustart}{\uend}$, 
contradicting the minimality of $m$.  In the former case, we may replace the original edge path $\mathcal{F}$ and adjustment sequence 
$\sigma$ by $\mathcal{F}''$ and $\sigma''$ and repeat the argument.  
Since $\tau_{k}$ was to the left of $\tau$ in $\mathcal{T}$ we may, after finitely 
many such replacements, reduce 
to the case $k<j$ and obtain the contradiction above.

The above argument also demonstrates the minimality of the value $m$ returned on line \ref{Code:ReturnYes}.  If $m$ is not minimal, 
there is a semi-canonical path $q$ for $p^{m'}$, with $m'$ minimal, that the algorithm \emph{does not find} (it searches 
breadth-first, and returns the first success). The failure cannot be due to the 
case $k<j$, since this implies that $m'$ is not minimal.  The failure must occur in the $k=j$ case, but we may then repeat the argument with 
a semi-canonical path $q'$ passing through the vertex $\tau_{k}$ to the left of $\tau$.  After finitely many steps we again reduce to the case 
$k<j$ and obtain a contradiction.

\end{proof}

By combining the \textsc{ReadPower} algorithm with $\A$-graph folding, we may solve the power coset membership problem in 
certain graphs of groups.

\begin{theorem}\label{Thm:PowerCosetMembership}
Let $G$ be the fundamental group of a benign graph of groups $\A$ in which 
every vertex group has decidable power 
coset membership problem, and such that for every folded $\A$-graph $\B$, the pair $(\A,\B)$ 
satisfies Property~\ref{Property:PowerReadingConditions}.  Then the power coset membership problem is decidable in $G$.
\end{theorem}
\begin{proof}
Let $v_{0}$ be the base vertex of $\A$, so that $G=\piA$.
Assume we are given as input to the power coset membership problem 
elements $x,p,h_{1},\ldots,h_{s}\in G$, expressed as $\A$-loops based at $v_{0}$.  The decidability of the membership problem
in vertex groups allows us to reduce paths, so we may assume that these paths are reduced.  Let $H=\langle h_{1}, \ldots, h_{s}\rangle$ and 
$p=\Path{p_{0}, e_{1}, p_{1}, e_{2}, \ldots,e_{n},p_{n}}$.

Suppose that $p$ is not cyclically reduced, that is, $p^{2}$ is not a reduced path.  
Then $e_{1}=e_{n}^{-1}$ and $p_{n}p_{0}=\omega_{e_{n}}(c)$ for some 
$c\in A_{e_{n}}$.  
Consider the conjugate 
\[
p^{p_{0}}=\Path{1,e_{1},p_{1},\ldots,e_{n},p_{n}p_{0}}\sim \Path{1,e_{1},p_{1},\ldots,p_{n-1}\alpha_{e_{n}}(c),e_{n},1}.
\]
Since $(p^{p_{0}})^{m}\in x^{p_{0}}H^{p_{0}}$ if and only if $p^{m}\in xH$, 
we may 
assume from the beginning that $p_{0}=p_{n}=1$.

Let $v_{1}=t(e_{1})$. 
Under the natural isomorphism $\piA\simeq \pi_{1}(\A,v_{1})$, the image of $p$ is the reduced path 
$p'=\Path{p_{1},e_{2},\ldots,e_{n-1},p_{n-1}}$ and we consider the same problem with $v_{1}$ as the base vertex instead of $v_{0}$
(replacing $x$ and $H$ by their images under this isomorphism).  
If $p'$ is not cyclically reduced, 
we may repeat this procedure, reducing the length of $p$ each time, until $p$ is cyclically reduced.  
Hence we may assume from the beginning that $p$ is cyclically reduced. We continue to denote the base vertex by $v_{0}$.

Since $\A$ is benign, we may apply Theorem~\ref{Thm:CosetGraph} with input $H$ and $x$.  
Suppose the algorithm reaches the second case, returning the coset graph $\Hx$.  
%Denote the vertex group of any vertex $u$ of $\Hx$ by $B_{u}$ and the edge group of 
%any edge $f$ by $B_{f}$.

The power coset membership problem has a positive answer if and only if  
there exists 
a path $q$ in $\Hx$ from $u_{x}$ to $u_{0}$ such that $\overline{\mu(q)}=\overline{p^{m}}$ or 
$\overline{\mu(q)}=\overline{(p^{-1})^{m}}$ for some $m>0$.  We can decide this using Theorem~\ref{Thm:ReadingPowers} 
(i.e. the \textsc{ReadPower} algorithm), since the pair $(\A, \Hx)$ satisfies 
Property~\ref{Property:PowerReadingConditions} by assumption.

Now suppose the algorithm of Theorem~\ref{Thm:CosetGraph} returns a path $y$ of length 0 
such that $yH=xH$. 
Construct the folded $\A$-graph $\B$ described in Theorem~\ref{Thm:Folding} 
and let $u_{0}$ be the base vertex of $\B$ and $B_{u_{0}}$ the associated vertex group.

Consider first the case when $p$ has length 0. Since $p^{m}$ and $y$ are both in the base vertex group $A_{v_{0}}$, 
$p^{m}\in yH$ if and only if 
$p^{m}\in y(H\cap A_{v_{0}})$. But $H\cap A_{v_{0}}=B_{u_{0}}$ (since $\mu$ preserves path length), hence the problem reduces 
to an instance of power coset membership in the vertex group $A_{v_{0}}$, which is decidable. 

Now consider the case when $p$ has non-zero length.  We must determine whether or not there exists a $\B$-path $q$ from $u_{0}$ to 
$u_{0}$ such that $\overline{\mu(q)}=\overline{y^{-1}p^{m}}$ or $\overline{\mu(q)}=\overline{y^{-1}(p^{-1})^{m}}$, for some $m>0$. 
We make a slight modification to the \textsc{ReadPower} algorithm, replacing the statement `$a:=p_{0}$' on line \ref{Code:a} by 
`$a:=y^{-1}p_{0}$' (and similarly for $p^{-1}$).  The modified algorithm clearly solves the problem.
\end{proof}

Note that the preconditions for Theorem~\ref{Thm:PowerCosetMembership}, in particular 
Property~\ref{Property:PowerReadingConditions}(\ref{PowerReading:DoubleCoset}), may be difficult to establish.  
In the next section, we 
will use quasi-convexity properties to establish this for groups obtained by iterated centralizer extensions, 
but let us mention here the simple case when edge groups of $\A$ are finite.

\begin{corollary}\label{Cor:FiniteEdgeGroups}
Let $G$ be the fundamental group of a graph of groups in which all edge groups are finite and all vertex groups have 
decidable power coset membership problem.  Then the power coset membership problem is decidable in $G$.
\end{corollary}
\begin{proof}
Property~\ref{Property:PowerReadingConditions}(\ref{PowerReading:DoubleCoset}) is immediate.
Since edge groups are finite, the various algorithmic problems involving edge groups reduce to finitely many word problems.  Finite groups are Noetherian, 
so the graph of groups is benign.
\end{proof}

%============================================================================================
%================== SUBGROUPS OF EXTENSIONS OF CENTRALIZERS ==============================
%============================================================================================

\section{Effective coherence, embedding, enumeration}\label{Section:EffectiveCoherence}
Our main goal in this section is to prove that every finitely generated group $G$ that is discriminated by a locally quasi-convex torsion-free hyperbolic 
group $\Gamma$ is effectively coherent.  We approach this via the fact that $G$ embeds into a group $G_{n}$ obtained from $\Gamma$ 
by iterated centralizer extensions.  Viewing a centralizer extension as the fundamental group of a graph of groups, we apply 
the folding algorithm to find the induced decomposition of, and hence a presentation for, 
a subgroup.  As applications, we compute the embedding of $G$ into $G_{n}$, and provide enumeration and 
recognition algorithms for groups discriminated by $\Gamma$.

An iterated centralizer extension corresponds to an iterated graph of groups construction.  Since we will need to 
prove that the graph of groups is benign, we begin in \S \ref{Subsection:IteratedBenign} by giving some general 
situations in which we can prove that an iterated 
graph of groups, with virtually cyclic edge groups, is benign.  

%only edge group is $\integers$, being `benign' reduces to the power coset membership problem being decidable in vertex groups.  
%We prove this by applying the results from 
%\S\ref{Section:GraphsOfGroups} inductively. 

\subsection{Benign iterated graphs of groups}\label{Subsection:IteratedBenign}
Consider a graph of groups $\A$ in which each vertex group is itself the fundamental group of a graph of groups. 
There are many important \emph{hierarchies} of this type, with various restrictions on the graph of groups and 
with the hierarchy terminating in base groups of a specific type 
(e.g. iterated centralizer extensions, finite hierarchies \cite{DP01} \cite{LT13}, quasi-convex hierarchies \cite{Wis12}).

When can 
we prove that $\A$ is benign?  If edge groups are cyclic, then conditions (\ref{BenignNoetherian}) and (\ref{BenignEdgeMembership}) 
of the definition of benign hold and 
(\ref{BenignCoset}) and (\ref{BenignGenSet}) 
reduce to instances of the power coset membership problem in vertex groups.  We will use Theorem \ref{Thm:PowerCosetMembership} 
to show that decidability of the power coset membership problem propagates up the hierarchy, but to do so we will need to prove, independently, 
that Property \ref{Property:DoubleCoset} holds at certain vertex groups.

For a vertex $u$, we say that Property~\ref{Property:DoubleCoset} \emph{holds at $u$ with respect to edge groups} if for every 
edge $e$ with $t(e)=u$ and every $a\in A_{u}$, $A_{u}$ satisfies Property~\ref{Property:DoubleCoset} with respect to $\omega_{e}(A_{e})^{a}$, 
and for every edge $e'\neq e^{-1}$ with $u=o(e')$ and every $a'\in A_{u}$, $A_{u}$ satisfies 
Property~\ref{Property:DoubleCoset} with respect to $(\omega_{e}(A_{u})^{a}, \alpha_{e'}(A_{u})^{a'})$.  Note that $e'$ and $e$ may 
coincide if $\A$ has loops.

\begin{theorem}\label{Thm:PowerCosetCombination}
Let $G$ be the fundamental group of a graph of groups $\A$ in which
\begin{enumerate}[(i)]
\item every edge group is virtually cyclic, \label{VirtuallyCyclic}
\item every vertex group has decidable power coset membership problem, and \label{DecidablePCM}
\item for every edge of $\A$, at least one of its adjacent vertex groups satisfies Property~\ref{Property:DoubleCoset} with respect to 
edge groups. \label{WithRespectToEdgeGroups}
\end{enumerate}
Then $\A$ is benign and $G$ has decidable power coset membership problem.
\end{theorem}
\begin{proof}
First, we check that $\A$ is benign. The membership problem is decidable in virtually cyclic groups 
(it follows, for example, from Lemma \ref{Lem:GammaPowerCoset} below since such groups are locally quasi-convex hyperbolic), 
so (\ref{BenignEdgeMembership}) holds.  
For (\ref{BenignNoetherian}), let $E$ be virtually cyclic with finite-index cyclic subgroup $\langle c\rangle$.  Note that 
every infinite subgroup of $E$ intersects $\langle c\rangle$ non-trivially, and that the torsion subgroup of $E$ is finite.  
If $E$ is not Noetherian, there exists an infinite proper chain of subgroups 
$H_{1}\lneq H_{2}\lneq\ldots$ of $E$.  Intersecting with $\langle c\rangle$ produces an infinite chain  
$H_{1}\cap \langle c\rangle \leq H_{2}\cap \langle c\rangle\ldots$ of subgroups of $\langle c\rangle$ which is 
proper at infinitely many steps, a contradiction.

For conditions (\ref{BenignCoset}) and (\ref{BenignGenSet}), let $X$ be a finite subset of a vertex group $A_{v}$, 
$a$ an element of $A_{v}$, and $E$ the image of an edge group in $A_{v}$.
Let $c$ be an element of infinite order in 
$E$ (such an element may be found since there is a computable bound on the order of finite elements in a hyperbolic group, 
see \cite{Bra00}).  The subgroup $\langle c\rangle$ must have finite index in $E$, and we may compute set of coset 
representatives $b_{1},\ldots, b_{m}$ (the membership problem is decidable, so we may construct the Schrier graph).

To find a generating set for $\langle X\rangle \cap E$, first find $k$ such that 
$\langle X\rangle \cap \langle c\rangle =\langle c^{k}\rangle$. 
Next, find all elements of the form $b_{i}c^{l}$, where $i\in \{1,\ldots,m\}$ and $0\leq l <k$, 
such that $b_{i}c^{l}\in \langle X\rangle$. 
These elements, together with $c^{k}$, form a generating set for $\langle X\rangle \cap E$.  To check if  
$\langle X\rangle \cap aE$ is non-empty, it suffices to check if $\langle X\rangle \cap (ab_{i})\langle c\rangle$ is non-empty 
for some $i\in\{1,\ldots,m\}$, and this may be decided using power coset membership in $A_{v}$. 

Decidability of the power coset membership problem in $G$ is a corollary of Theorem \ref{Thm:PowerCosetMembership}.  Indeed, 
for any pair $(\A,\B)$ with $\B$ folded, Property \ref{Property:PowerReadingConditions}(\ref{PowerReading:VertexMembership}) holds 
by assumption, Property \ref{Property:PowerReadingConditions}(\ref{PowerReading:Intersection}) holds since $\A$ is benign, and 
Property \ref{Property:PowerReadingConditions}(\ref{PowerReading:FiniteIndex}) holds since edge groups are virtually cyclic.
For Property \ref{Property:PowerReadingConditions}(\ref{PowerReading:DoubleCoset}), consider any edge cycle $f_{1},\ldots,f_{s}$ 
of $\B$ (the assumption that all $B_{f_{i}}$ are finite will not be needed). If $s=1$ then $e_{1}$ is a loop hence 
Property \ref{Property:DoubleCoset} holds at $o(e_{1})$ by assumption and case (1) of 
Property \ref{Property:PowerReadingConditions}(\ref{PowerReading:DoubleCoset}) is satisfied. 
If $s>1$, consider the edge $f_{2}$.  Either Property \ref{Property:DoubleCoset} holds at
$o(e_{2})$ so Property \ref{Property:PowerReadingConditions}(\ref{PowerReading:DoubleCoset}) 
is satisfied for $i=1$, or Property \ref{Property:DoubleCoset} 
holds at $t(e_{2})$ so Property \ref{Property:PowerReadingConditions}(\ref{PowerReading:DoubleCoset}) is satisfied for $i=2$.
%If $e_{2}\neq e_{1}^{-1}$, then either Property \ref{Property:DoubleCoset} holds at
%$o(e_{2})$ so Property \ref{Property:PowerReadingConditions}(\ref{PowerReading:DoubleCoset}) 
%is satisfied for $i=1$, or Property \ref{Property:DoubleCoset} holds at $t(e_{2})$ and 
%Property \ref{Property:PowerReadingConditions}(\ref{PowerReading:DoubleCoset}) is satisfied for $i=2$.

\end{proof}

This theorem requires that Property \ref{Property:DoubleCoset} holds for some vertex groups, but we do not have 
a method to prove this property based purely on graph of groups considerations.  Instead, we show that 
certain locally quasi-convex relatively hyperbolic groups satisfy the property. 

There are several \emph{combination theorems} \cite{Dah03} that give conditions under which  
the fundamental group $G$ of a graph of groups with (relatively) hyperbolic vertex groups is itself relatively hyperbolic, 
and a recent theorem of Bigdely and Wise \cite{BW13} shows that local quasi-convexity of the vertex  groups often implies  
local relative quasi-convexity of $G$.  
Therefore relatively hyperbolic groups that are locally relatively quasi-convex arise naturally in graph of groups constructions. 
%so 
%we will prove that Property \ref{Property:DoubleCoset} holds in such groups.
 
We recall the definition of local quasi-convexity.
%We recall some notions of (relative) quasi-convexity and relative hyperbolicity. 
Let $G$ be a group generated by a finite set $X$.  Let $\mathcal{X}=\mathrm{Cay}(G, X)$ be the Cayley graph of $G$ with respect to $X$, 
and $d_{X}$ the associated metric.  
For a path $\gamma$ in $\mathcal{X}$ we 
denote by $|p|$ the length of $p$ and by $o(\gamma)$ and $t(\gamma)$ the initial and terminal vertices of $\gamma$, respectively. 
A path $p$ is a 
\emph{$(\lambda, c)$-quasi-geodesic} if for every subpath $q$ of $p$, $|q|\leq \lambda d_{X}(o(q), t(q))+c$. 

A subgroup $R\leq G$ is called \emph{quasi-convex} if there exists a constant $\epsilon\geq 0$ (`quasi-convexity constant') 
such that the following holds: 
for every pair of elements $r_{1}, r_{2}\in R$ and every geodesic 
$\gamma$ with $o(\gamma)=r_{1}$ and $t(\gamma)=r_{2}$, every vertex of $\gamma$ is within distance $\epsilon$ of a 
vertex belonging to $R$. 
Though quasi-convexity of a subgroup may depend on the choice of generating set $X$, when $G$ is hyperbolic it is independent of the choice 
of \emph{finite} generating set (\cite{CDP90} Prop. 10.4.1).

If every finitely generated subgroup of $G$ is quasi-convex then $G$ is called \emph{locally quasi-convex}. 
Examples of locally quasi-convex groups include free groups, (most) surface groups, 
and a variety of small-cancellation, Coxeter, and one-relator groups recently studied by Mart\'{i}nez-Pedroza, McCammond, and Wise
(\cite{MW05}, \cite{MW08}, \cite{MW11Local}).  We will use these groups as vertex groups in Theorem \ref{Thm:PowerCosetCombination}, so 
decidability of the poset coset membership problem is essential. 

\begin{lemma}\label{Lem:GammaPowerCoset}
Every locally quasi-convex hyperbolic group $\Gamma$ has decidable power coset membership problem.
\end{lemma}
\begin{proof}
Let $H\leq \Gamma$ be finitely generated, 
and let $x,g\in \Gamma$. 
Compute a hyperbolicity constant $\delta$ for $\Gamma$ and a quasi-convexity constant $\epsilon$ for $H$ (using 
\cite{Ol91} or \cite{Pap96} and \cite{Kap96} Prop.~4).  First, determine if $g$ has finite order.  
Since the order of any finite-order element of $\Gamma$ is bounded by $s^{2\delta +1}+1$, where $s$ is the cardinality of the given 
generating set of $\Gamma$ (see \cite{Bra00}), it suffices to 
check if $g^{i}$ is trivial for some $0<i\leq s^{2\delta +1}+1$.  If $g$ has finite order $k$, check if $g^{i}\in xH$ for any 
$0<i<k$, using the fact that membership in a quasi-convex subgroup of a hyperbolic group is decidable (\cite{Gru99Angle} Prop.~6.1).
Assume now that $g$ has infinite order.
 
Suppose that $g^{m}\in xH$ for some $m>0$.  In the Cayley graph of $\Gamma$, 
the path labelled by $g^{m}$ is a $(\lambda, c)$-quasi-geodesic, for some $\lambda$ and $c$ depending on $\delta$ and $|g|$ but not on $m$ 
(\cite{Ols93} Lem.~1.11). Consequently, the path $q$ labelled by $x^{-1}g^{m}$ is a $(\lambda, c')$-quasi-geodesic, where 
$c'=c+(\lambda+1)|x|$.  
Let $p$ be a geodesic path from 1 to $x^{-1}g^{m}$.  
There exists a computable constant $K=K(\delta, \lambda, c')$ such that every vertex of $q$ is 
within distance $K$ of $p$ (\cite{Ols93} Lem.~1.9).  Since $H$ is 
$\epsilon$-quasi-convex, every vertex of $p$ lies within $\epsilon$ of $H$.  Hence every vertex of $q$ lies within distance $K+\epsilon$ of $H$.
In particular, there exist elements $y_{0}, y_{1},\ldots, y_{m}$ in the ball $B$ of radius $K+\epsilon$ centered at 1 such that 
\[
x^{-1}g^{i}y_{i}\in H,
\]
for $i=0,\ldots,m$ (see Figure \ref{Fig:LocallyQC}).

\begin{figure}[h]
\begin{center}
\includegraphics[scale=0.9]{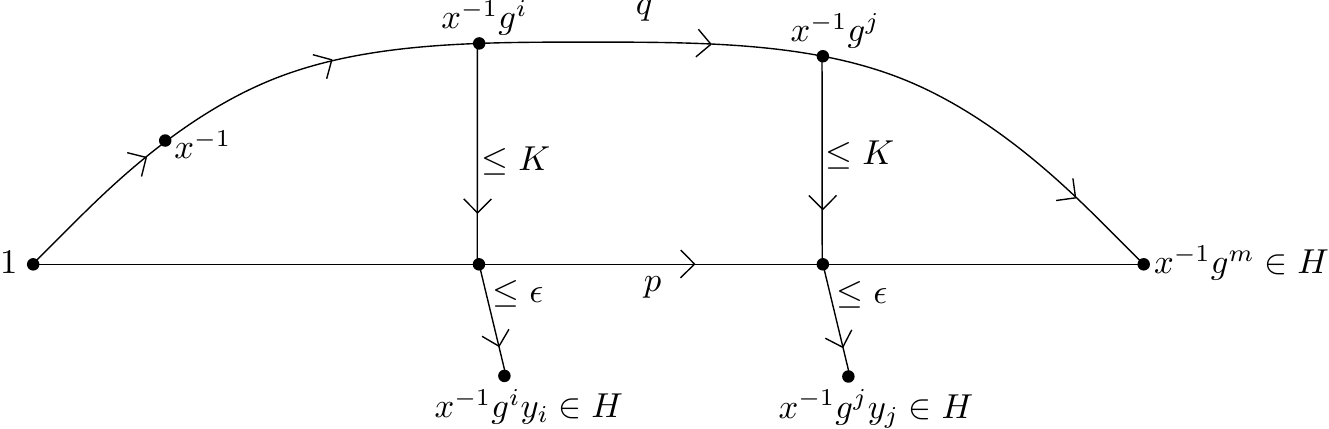}
\caption{vertices of $q$ are $(K+\epsilon)$-close to $H$}
\label{Fig:LocallyQC}
\end{center}
\end{figure}

Let $N=|B|$.  We claim that there exists $0<n\leq N$  
such that $x^{-1}g^{n}\in H$. If $m\leq N$, the claim holds, so assume $m>N$. 
Then there exist $0\leq i < j < m$ such that $y_{i}=y_{j}$.  %Let $x=x_{i}$.

Let $h_{1}=x^{-1}g^{i}y_{i}$, $h_{2}= (y_{i}^{-1}g^{-i}x)x^{-1}g^{j}y_{i}$, 
and $h_{3}= (y_{i}^{-1}g^{-j}x) x^{-1}g^{m}$, which are all elements of $H$. 
%As elements of $\Gamma$, we have that $h_{2}=y_{i}^{-1}g^{j-i}y_{i}$ and $h_{3}=y_{i}^{-1} g^{m-j}$.  Hence 
Then 
\begin{eqnarray*}
h_{1} h_{3} & = & h_{1} y_{i}^{-1} g^{m-j} = h_{1} (y_{i}^{-1} g^{j-i} y_{i})(y_{i}^{-1} g^{m-j})g^{-(j-i)} \\
	& = & h_{1} h_{2} h_{3} g^{-(j-i)} = x^{-1}g^{m-(j-i)}
\end{eqnarray*}
so $x^{-1}g^{m-(j-i)}\in H$. If $m-(j-i)\leq N$, the claim holds, and otherwise we repeat the above argument with $m'=m-(j-i)$ in place of $m$ 
(note that $m'>0$). This proves 
the claim.

If $g^{m}\in xH$ for some $m<0$, a similar argument shows that $x^{-1}g^{n}\in H$ for some $-N\leq n<0$.  We conclude that 
if there exists $m\neq 0$ such that $g^{m}\in xH$, then there exists $n\neq 0$ in the interval $[-N,N]$ such that $g^{n}\in xH$.  
Then to solve the power coset membership problem, 
it suffices to check if any of the elements $x^{-1}g^{i}$ is in $H$, for $i\in[-N,N]\setminus\{0\}$. 
\end{proof}

%\begin{definition}
%Let $\mathcal{A}$ be the class of all groups that are Noetherian, effectively coherent, and have 
%decidable power coset membership problem.
%\end{definition}
We now review some necessary aspects of the geometry of relatively hyperbolic groups from \cite{Osi06Relatively}.
Let $G$ be hyperbolic relative to a collection $\mathcal{H}=\{H_{1},\ldots,H_{m}\}$ of subgroups. 
Subgroups of $G$ that are conjugate into subgroups in $\mathcal{H}$ are called \emph{parabolic}.
Denote $\hat{\mathcal{H}} = \sqcup_{s=1}^{m} H_{s}\setminus\{1\}$ let 
$\hat{\mathcal{X}}=\mathrm{Cay}(G, X\cup \hat{\mathcal{H}})$.
%, and denote the associated metric by $d_{X\cup\hat{\mathcal{H}}}$. 
When constructing $\hat{\mathcal{X}}$ we remove from $X$ any parabolic elements.

For a path $p$ in $\hat{\mathcal{X}}$, a maximal subpath 
consisting of edges from $H_{i}$ is called an \emph{$H_{i}$-component}. 
Every vertex that does not lie in the interior of some $H_{i}$-component is called a \emph{phase vertex}.
Two $H_{i}$-components $p_{1}$ and $p_{2}$ are \emph{connected} if 
there is an edge from a vertex of $p_{1}$ to a vertex of $p_{2}$ labelled by an element of $H_{i}$.  An $H_{i}$-component is 
\emph{isolated} if it is not connected to any other $H_{i}$-component.

A subgroup $R$ of $G$ is called \emph{relatively quasi-convex} 
if there exists a constant $\epsilon\geq 0$ such that 
the following condition holds: for every two elements $r_{1}, r_{2}\in R$, every geodesic path $\gamma$ in 
$\hat{\mathcal{X}}$ with $o(\gamma)=r_{1}$ and $t(\gamma)=r_{2}$, and every vertex $v$ of $\gamma$, 
there exists a vertex $w\in R$ such that 
\[
d_{X} (v, w) \leq \epsilon.
\]
Note that this distance is in $\mathcal{X}$, not $\mathcal{\hat{X}}$.
As with quasi-convexity, relative quasi-convexity (in a relatively hyperbolic group) does not depend on the finite generating set $X$ (\cite{Osi06Relatively} Prop. 4.10). If every finitely generated subgroup of $G$ is relatively quasi-convex then $G$ is 
called \emph{locally relatively quasi-convex}.

\begin{lemma}\label{Lem:DoubleCoset}
Let $G$ be relatively hyperbolic and locally relatively quasi-convex.  Then for every hyperbolic element $c\in G$ 
of infinite order, $G$ satisfies Property~\ref{Property:DoubleCoset} with respect to the maximal elementary subgroup $E$ 
containing $c$.

%For every $c\in G$ such that 
%$C(c)=\langle c\rangle$, 
%$G$ satisfies Property~\ref{Property:DoubleCoset} with respect to $\langle c\rangle$.
\end{lemma}

\begin{proof}
Let $s$ be the index of $C=\langle c\rangle$ in $E$ and let $A$ be a set of coset representatives 
of $C$ in $E$.
Let $g,x\in G$ with $gx^{-1}\not\in E$ and let $K\leq G$ be finitely generated such that 
$K\cap E$ and $K\cap E^{x}$ are both finite.  Note that $C^{x}$ has index $s$ in $E^{x}$ and that $A^{x}$  
is a set of coset representatives $C^{x}$ in $E^{x}$. 

Let $I$ denote the set of all quadruples $(a_{1},a_{2},n,l)$ with $a_{1},a_{2} \in A$ 
and $n,l\in \integers$ such that  
\begin{equation}
a_{1} c^{n} g a_{2}^{x}(c^{x})^{l} \in K.
\end{equation} 
%Since every element of $E$ has the form $a_{i}c^{n}$ and every element of $E^{x}$ has the form 
%$a_{j}^{x} (c^{x})^{l}$, 
To prove Property~\ref{Property:DoubleCoset} we must show that $I$ is finite.  Assume for 
contradiction that $I$ is infinite.
%We denote by $I$ the set of all quadruples 
%$(i,j,n,l)\in \integers^{4}$ such that 
%\begin{equation}
%e_{i}c^{n} g e_{j}^{x}(c^{x})^{l} \in H.
%\end{equation} 

First, suppose that there are more than $s^{2}$ elements of $I$ that have identical third components.  
Then there exists $(a_{1},a_{2},n,l)$ and $(a_{1},a_{2},n,k)$ in $I$ with $l\neq k$.
Then   
\[
(a_{1} c^{n} g a_{2}^{x}(c^{x})^{l})^{-1}(a_{1} c^{n} g a_{2}^{x}(c^{x})^{k}) = (c^{x})^{k-l}\in K,
\]
hence $\langle (c^{x})^{k-l}\rangle \leq K$ contradicting the fact that $K\cap E^{x}$ is finite.  
Similarly, if there are more than $s^{2}$ elements of $I$ having identical fourth components, one obtains a contradiction 
to the fact that $K\cap E$ is finite.
It follows then that there exists $a_{1},a_{2}\in A$ and 
an infinite subset $J\subset I$ indexed by $\naturals$ consisting of elements 
$(a_{1},a_{2},n_{i},k_{i})$ where $n_{i}\neq n_{j}$ and $k_{i}\neq k_{j}$ for all $i\neq j$.

%If there exist two such pairs $(n,k)$ and $(m,k)$ with $n\neq m$, then 
%\[
%c^{n} g (c^{x})^{k} (c^{x})^{-k} g^{-1} c^{-m} = c^{n-m}\in H
%\]
%contradicting $H\cap \langle c\rangle =1$.  

%for each $i\in \naturals$ two integers 
%$n_{i}$ and $k_{i}$ such that 
%\begin{equation}\label{Eqn:PairsInH}
%a c^{n_{i}} g b^{x} (c^{x})^{k_{i}}\in H,
%\end{equation}
%and such that $n_{i}\neq n_{j}$ and $k_{i}\neq k_{j}$ for all $i\neq j$.

Let $G$ be hyperbolic relative to $\mathcal{H}$.  Then $G$ is also hyperbolic relative to 
$\mathcal{H}'=\mathcal{H}\cup \{E\}$, since by \cite{Osi06Elementary} Corollary 1.7 the maximal elementary 
subgroup containing a hyperbolic element is always hyperbolically embedded (if a conjugate of $E$ is already included we simply 
replace it by $E$).  Let 
$\hat{\mathcal{X}}=\mathrm{Cay}(G, \hat{\mathcal{H}'})$.
Edges of $\hat{\mathcal{X}}$ labelled by elements of $E$ will be called $E$-edges.
%Note that every vertex 
%of $\hat{\mathcal{X}}$ has an outgoing edge $e^{(n)}$, for all $n$.

Let $p_{gx^{-1}}$ be a geodesic path in $\hat{\mathcal{X}}$ representing the element $gx^{-1}$. Write $p_{gx^{-1}}$ in the form 
\[
p_{gx^{-1}}=e \hat{p}_{gx^{-1}} e'
\]
where $e$ and $e'$ are $E$-edges labelled by elements $b,b'\in E$ and $\hat{p}_{gx^{-1}}$ neither begins nor ends with an $E$-edge.  
Since $gx^{-1}\not\in E$, $\hat{p}_{gx^{-1}}$ is non-empty.  
Since $p_{gx^{-1}}$ is a geodesic, every $E$-component of $p_{gx^{-1}}$ is isolated and consists of a 
single edge.  
Let $p_{x}$ be a geodesic path for $x$ and write $p_{x}$ in the form 
\[
p_{x} = e'' \hat{p}_{x}
\]
where $e''$ is an $E$-edge labelled by $b''\in E$ and $\hat{p}_{x}$ does not begin with an $E$-edge.  
Every $E$-component of $p_{x}$ is isolated and consists of a single edge.  

For every $i\in \naturals$, let $p_{i}$ be the path from 1 to $a_{1} c^{n_{i}} gx^{-1}a_{2} c^{k_{i}}x$ defined by 
\[
p_{i} = e^{(1)}_{i} \hat{p}_{gx^{-1}} e^{(2)}_{i} \hat{p}_{x},
\]
where $e^{(1)}_{i}$ is the $E$-edge labelled by $a_{1}c^{n_{i}}b$ and $e^{(2)}_{i}$ is the $E$-edge labelled by $b' a_{2}c^{k_{i}}b''$.
The path $p_{i}$ is shown in Figure \ref{Fig:RelativelyQC}.

\begin{claim}
For infinitely many $i$, $p_{i}$ is a path without backtracking, i.e. for all $H\in \mathcal{H}$, 
every $H$-component of $p_{i}$ is isolated.  
\end{claim}
\begin{proof}
First we show that every $E$-component is isolated.  Since all $n_{i}$ are distinct, as are all $k_{i}$, there are 
infinitely many $i$ such that $a_{1}c^{n_{i}}b\neq 1$ and $b'a_{2}c^{k_{i}}b''\neq 1$ so we may assume these two elements 
are non-trivial.  
%Consequently, 
%as are all $k_{i}$, there are 
%at most two values of $i$ for which either  or $k_{i}'=0$, so we may ignore these cases and assume that $n_{i}'\neq 0$ and 
%$k_{i}'\neq 0$. 
Hence the $E$-components of $p_{i}$ are precisely the edges $e^{(1)}_{i}$ and $e^{(2)}_{i}$ and 
the $E$-edges appearing in $\hat{p}_{gx^{-1}}$ and in $\hat{p}_{x}$.  

Assume for contradiction that $p_{i}$ has a non-isolated 
$E$-component.  The component $e^{(1)}_{i}$ is not connected to an $E$-component 
of $\hat{p}_{gx^{-1}}$, as this would imply that either $\hat{p}_{gx^{-1}}$ is not a geodesic or that $\hat{p}_{gx^{-1}}$ begins with an $E$-edge. 
Similarly, $e^{(2)}_{i}$ is neither connected to an $E$-component of 
$\hat{p}_{gx^{-1}}$ nor to an $E$-component of $\hat{p}_{x}$.  
If $e^{(1)}_{i}$ and $e^{(2)}_{i}$ are connected, it implies that $\hat{p}_{gx^{-1}}$ represents 
an element of $E$, hence $gx^{-1}\in E$, which is false.  

The only remaining possibility is that $\hat{p}_{x}$ has an $E$-component which is connected to either $e^{(1)}_{i}$ or 
to an $E$-component of $\hat{p}_{gx^{-1}}$. We claim that this occurs for only finitely many $i$.  Assume otherwise. Then for infinitely many $i$, 
there exists $1\leq l_{i}\leq |\hat{p}_{gx^{-1}}|$ and $2\leq m_{i}\leq |\hat{p}_{x}|+1$ such that the 
vertex number $l_{i}$ of 
$\hat{p}_{gx^{-1}}$ (where vertex number 1 is the first vertex from the left) is 
connected via an $E$-edge to vertex number $m_{i}$ of $\hat{p}_{x}$. Since $\hat{p}_{gx^{-1}}$ and $\hat{p}_{x}$ are finite, 
there exist $i\neq j$ such that $l_{i}=l_{j}$ and $m_{i}=m_{j}$.  It follows that 
\begin{equation}\label{Eqn:InC}
g_{1} (b'a_{2}c^{k_{i}}b'') g_{2}\in E,\;\;\; g_{1} (b' a_{2}c^{k_{j}}b'') g_{2}\in E
\end{equation}
where $g_{1}, g_{2}\in G$ are the elements represented by the length $|\hat{p}_{gx^{-1}}|-l_{i}+1$ terminal segment of $\hat{p}_{gx^{-1}}$  
the length $m_{i}-1$ initial segment of $\hat{p}_{x}$ (respectively).  Taking the difference of the elements appearing in (\ref{Eqn:InC}), we conclude 
that 
\[
(c^{k_{j}-k_{i}})^{b'' g_{2}}\in E.
\]
Since $E$ is hyperbolically embedded in $G$, it almost malnormal by Theorem 1.5(3) of \cite{Osi06Elementary}.  That is, 
$E\cap E^{h}$ is finite for all $h\in G\setminus E$.
Since $\langle c^{k_{j}-k_{i}}\rangle^{b''g_{2}}$ is an infinite subgroup of $E\cap E^{b'' g_{2}}$, 
the conjugator $b''  g_{2}$ must lie in $E$, which implies $g_{2}\in E$.
However, $g_{2}$ corresponds to the length $m_{i}-1$ initial segment of $\hat{p}_{x}$.  
This segment cannot be length 1, as $\hat{p}_{x}$ does not begin with a $E$-edge, 
nor can it be length greater than 1, as $\hat{p}_{x}$ is a geodesic.  It must therefore be length 0, contradicting $m_{i}\geq 2$.

Now we check that for each $H\in \mathcal{H}'$, where $H\neq E$, every $H$-component in $p_{i}$ is isolated for all but finitely 
many $i$. Assume otherwise.
The $H$-components may only occur in $\hat{p}_{gx^{-1}}$ and $\hat{p}_{x}$.  Since each of these paths 
is a geodesic, no two $H$-components of $\hat{p}_{gx^{-1}}$ 
are connected, and similarly for $\hat{p}_{x}$. Hence an $H$-component of $\hat{p}_{gx^{-1}}$ is connected to an 
$H$-component of $\hat{p}_{x}$, for infinitely many 
$i$. As in the case for $E$-components, we conclude there exists $i\neq j$, $g_{1}\in G$ corresponding to a terminal segment of 
$\hat{p}_{gx^{-1}}$, and $g_{2}\in G$ corresponding to an initial segment of $\hat{p}_{x}$ such that 
\[
g_{1} (b'a_{2}c^{k_{i}}b'') g_{2}\in E,\;\;\; g_{1} (b' a_{2}c^{k_{j}}b'') g_{2}\in E.
\]
It follows that $(c^{k_{j}-k_{i}})^{b'' g_{2}}\in H$, hence $\langle c^{k_{j}-k_{i}}\rangle^{b''g_{2}}$ is an infinite subgroup 
of $E^{b'' g_{2}}\cap H$.  But $H$ and $E$ are both in $\mathcal{H}'$, so
the intersection $E^{b'' g_{2}}\cap H$ must be finite (see Prop. 2.36 of \cite{Osi06Relatively}), which 
is a contradiction.
\end{proof}

We return to the proof of Lemma~\ref{Lem:DoubleCoset}. The length of $p_{i}$ 
is bounded by 
\[
\gamma = 2 + 2|x| + |g|,
\]
so $p_{i}$ is, trivially, a $(1,\gamma)$-quasi-geodesic.  Let $q_{i}$ be a geodesic in $\hat{\mathcal{X}}$ 
from 1 to $a_{1}c^{n_{i}} gx^{-1} a_{2} c^{k_{i}} x$. Since 
$p_{i}$ and $q_{i}$ are both quasi-geodesic paths with the same endpoints and $p_{i}$ is a path without backtracking, 
there exists a constant $\sigma=\sigma(\gamma)$ such that every (phase) vertex of $p_{i}$ is within $d_{X}$-distance $\sigma$ of a vertex of $q_{i}$ 
(\cite{Osi06Relatively} Prop. 3.15). In particular, for the endpoint $u_{i}$ of 
$e^{(1)}_{i}$, there exists a vertex $v_{i}$ in $q_{i}$ at $d_{X}$-distance at most $\sigma$ from $u_{i}$. 
Let $\epsilon$ be a quasi-convexity constant for $K$.  Since $q_{i}$ is a geodesic in $\hat{\mathcal{X}}$ between elements of $K$, 
$v_{i}$ is at $d_{X}$-distance at most $\epsilon$ 
from a vertex belonging to $K$.  Hence for infinitely many $i$, $u_{i}$ is connected in $\mathcal{X}$ to a vertex of 
$K$ by a path $r_{i}$ of $d_{X}$-length at most $\sigma+\epsilon$, as shown in Figure \ref{Fig:RelativelyQC}.

\begin{figure}
\begin{center}
\includegraphics[scale=.8]{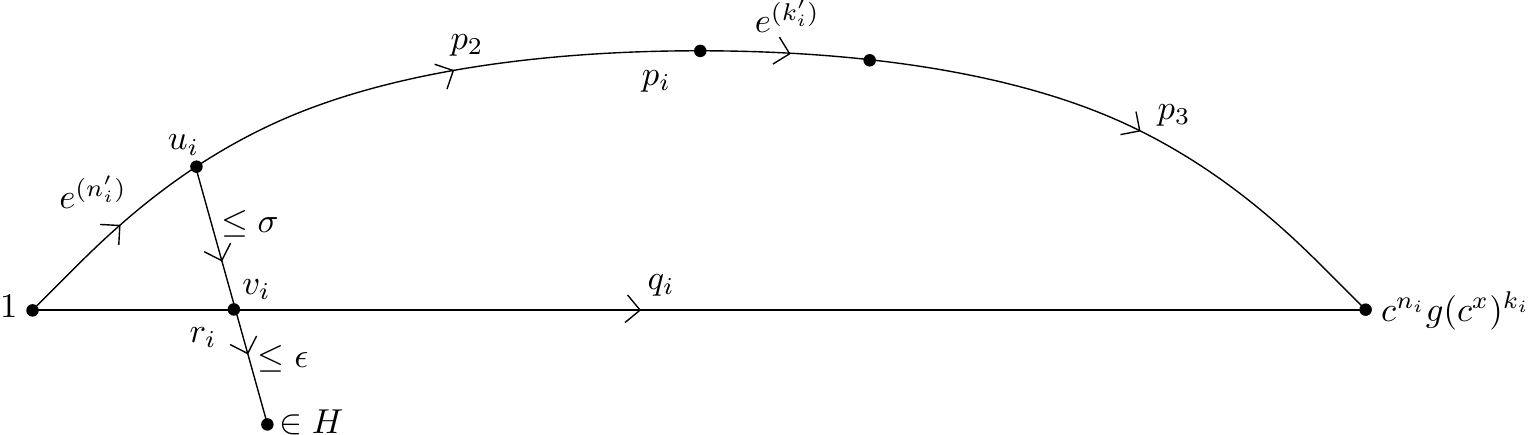}
\caption{a path $r_{i}$ of length at most $\sigma+\epsilon$ connects $u_{i}$ to $H$}
\label{Fig:RelativelyQC}
\end{center}
\end{figure}

Since the ball of radius $\sigma+\epsilon$ is finite in $\mathcal{X}$, there exist $i\neq j$ such that $r_{i}$ and $r_{j}$ are labelled 
by the same element $g'\in G$.  We conclude that 
\[
a_{1} c^{n_{i}}b g'\in K,\;\;\; a_{1} c^{n_{j}}b g'\in K
\]
hence $(c^{n_{i}-n_{j}})^{a_{1}^{-1}}\in K$, so $\langle (c^{n_{i}-n_{j}})^{a_{1}^{-1}}\rangle$ is an infinite subgroup of $K\cap E$.
This is a contradiction, so $I$ must be finite as required.
\end{proof}

\subsection{Groups discriminated by $\Gamma$}
Let $\Gamma$ be a torsion-free locally quasi-convex hyperbolic group and $G$ a finitely presented 
group discriminated by $\Gamma$.  Since $G$ embeds into an iterated centralizer extension of $\Gamma$, we will 
prove effective coherence of $G$ by first proving effective coherence in centralizer extensions of $\Gamma$. This will also 
allow us to compute the embedding.

An \emph{extension of a centralizer of $\Gamma$} is a group $G_{1}$ presented by 
\[
G_{1} = \langle \Gamma, t_{1},\ldots,t_{r} \; | \; [C(g), t_{i}]=[t_{i},t_{j}]=1, \; 1\leq i,j \leq r\rangle
\]
where $g\in \Gamma$ and $C(g)$ is the centralizer of $g$ in $\Gamma$.  Note that $G_{1}$ is isomorphic to the 
amalgamated product $\Gamma\ast_{C(g)} (C(g)\times\integers^{r})$, which we regard as the fundamental group of a graph of 
groups with two vertices.

The operation of forming an extension of a centralizer may be iterated to form a chain
\begin{equation}\label{Eqn:Chain}
\Gamma = G_{0} < G_{1} < \ldots  < G_{n}
\end{equation}
where $G_{i+1}$ is an extension of a centralizer of $G_{i}$.  We say that $G_{n}$ is obtained from $\Gamma$ by iterated extensions of centralizers.

For every non-trivial $g\in \Gamma$, its centralizer $C(g)$ is cyclic and malnormal, and $g$ may be chosen so that 
$C(g)=\langle g\rangle$. It follows (see \cite{MR96}) that all centralizers 
in $G_{n}$ are torsion-free abelian groups of finite rank and that the chain (\ref{Eqn:Chain}) may be arranged so that each 
$G_{i+1}$ is obtained from $G_{i}$ by extension of a cyclic centralizer, i.e.
\begin{equation}\label{Eqn:GiTF}
G_{i+1} = G_{i}\ast_{\langle g_{i}\rangle}\integers^{r_{i}+1}.
\end{equation} 
We will apply combination theorems to obtain relative hyperbolicity and local relative quasi-convexity of each $G_{i+1}$, then apply 
Theorem \ref{Thm:PowerCosetCombination} inductively to obtain effective coherence of $G_{n}$.
This is the central theorem of the paper.

\begin{theorem}\label{Thm:EffectiveCoherenceExtensions}
There is an algorithm that, given
\begin{itemize}
\item a finitely presented group $\Gamma$ that is hyperbolic, torsion-free, and locally quasi-convex,
\item a chain of centralizer extensions $\Gamma=G_{0} < G_{1} < \ldots < G_{n}$
and,
\item a finite subset $X\subset G_{n}$,
\end{itemize}
produces a presentation for the subgroup $\langle X\rangle$ generated by $X$.
%, provided $\langle X\rangle$ is relatively quasi-convex in $G_{n}$. 
%In particular, if $\Gamma$ is locally quasi-convex then algorithm always terminates producing a presentation for $\langle X\rangle$.
\end{theorem}

The absence of torsion in $\Gamma$ does not play an essential role in the proof of Theorem \ref{Thm:EffectiveCoherenceExtensions}, so 
we will prove it in a more general case, starting from a base group $\Gamma$ that is hyperbolic and locally quasi-convex, but possibly 
with torsion.  
The iterated centralizer extension construction must be modified. As before, we will construct a chain of groups
\begin{equation}\label{Eqn:ChainElementary}
\Gamma = G_{0} < G_{1} < G_{2} < \ldots < G_{n}.
\end{equation}
To construct $G_{1}$, choose $g_{0}\in\Gamma$ to be 
an element of infinite order and 
form an amalgamated product 
\[
G_{1} = \Gamma \ast_{E(g_{0})=E_{0}} V_{0}
\]
where $V_{0}$ is a finitely presented virtually abelian group and $E(g_{0})$ is the maximal elementary
subgroup containing $g_{0}$.  The group $E(g_{0})$ is 
hyperbolically embedded in $\Gamma$ (see \cite{Osi06Elementary} Cor. 1.7), 
meaning $\Gamma$ is hyperbolic relative to $E(g_{0})$.  Then it follows that 
$G_{1}$ is hyperbolic relative to 
$V_{0}$ (see for example \cite{Dah03}) and locally relatively quasi-convex (see \cite{BW13} Thm. 3.1).

In general, $G_{i+1}$ is obtained from $G_{i}$ by choosing a hyperbolic element $g_{i}\in G_{i}$ of infinite order, 
a finitely presented virtually 
abelian group $V_{i}$, a subgroup $E_{i}\leq V_{i}$ isomorphic to $E(g_{i})$, and forming the amalgamated product
\begin{equation}\label{Eqn:Gi}
G_{i+1} = G_{i}\ast_{E(g_{i})=E_{i}} V_{i}.
\end{equation}
The same argument as above shows that $G_{i+1}$ is relatively hyperbolic and locally relatively 
quasi-convex.  

%Note that the centralizer $C(g_{i})$ is contained in $E(g_{i})$, so this construction is a generalization 
%of iterated centralizer extensions. 
In the case 
when $\Gamma$ is torsion-free, each group $G_{i}$ is toral relatively hyperbolic.  Since each $g_{i}$ may be chosen in (\ref{Eqn:GiTF}) so that 
$C(g_{i})=\langle g_{i}\rangle$, one may assume that $g_{i}$ is a hyperbolic element.  Then the elementary subgroup $E(g_{i})$ 
coincides with the centralizer $C(g_{i})$.  
Thus Theorem \ref{Thm:EffectiveCoherenceExtensions} follows from 
Theorem \ref{Thm:EffectiveCoherenceElementary} below.

\begin{theorem}\label{Thm:EffectiveCoherenceElementary}
There is an algorithm that, given 
\begin{itemize}
\item a finitely presented group $\Gamma$ that is hyperbolic and locally quasi-convex,
\item a chain of amalgamated products 
\[
\Gamma=G_{0} < G_{1} < \ldots < G_{n}
\]
where $G_{i+1}=G_{i}\ast_{E(g_{i})=E_{i}} V_{i}$ for $i=0,\ldots,n-1$ with $g_{i}\in G_{i}$ hyperbolic of infinite order and $V_{i}$ virtually abelian, and  
%$A_{i}\leq V_{i}$ a specified free abelian subgroup of finite index, and 
%where $G_{i+1}=G_{i}\ast_{E(g_{i})=E_{i}} V_{i}$ for $i=0,\ldots,n-1$ with $g_{i}\in G_{i}$ hyperbolic of infinite order and $V_{i}$ virtually abelian with 
%$A_{i}\leq V_{i}$ a specified free abelian subgroup of finite index, and
% as in (\ref{Eqn:Gi}), and  
%and a free abelian subgroup $A_{i}$ of finite rank for each $V_{i}$, and
\item a finite subset $X\subset G_{n}$,
\end{itemize}
produces a presentation for the subgroup $\langle X\rangle$ generated by $X$.
\end{theorem}

\begin{proof}
In the input we assume that for each $i$, the following data is provided: generating sets for $E(g_{i})$ and $E_{i}$, 
a bijection between those generating sets inducing an isomorphism 
$E(g_{i})\simeq E_{i}$, a presentation for $V_{i}$, and a free basis for a free abelian subgroup $A_{i}\leq V_{i}$ of finite index.

We regard each $G_{i+1}$ as the fundamental group of a graph of groups with two vertex groups $G_{i}$, $V_{i}$ and one edge group $E(g_{i})$.
We prove, by induction on $n$, that each $G_{i+1}$ has decidable power coset membership problem and the corresponding graph of groups 
is benign by applying Theorem \ref{Thm:PowerCosetCombination} at each step.

In the base case we have $G_{1}=\Gamma\ast_{E(g_{0})=E_{0}}V_{0}$.  The edge group $E(g_{0})$ is virtually cyclic, and the vertex 
group $\Gamma$ satisfies  
Property \ref{Property:DoubleCoset} with respect to $E(g_{0})$ and all of its conjugates $E(g_{0}^{x})=E(g_{0})^{x}$ by 
Lemma \ref{Lem:DoubleCoset}, so conditions (\ref{VirtuallyCyclic}) and (\ref{WithRespectToEdgeGroups}) 
of Theorem \ref{Thm:PowerCosetCombination} hold. 
The vertex group $\Gamma$ has decidable power coset membership problem 
by Lemma \ref{Lem:GammaPowerCoset}, so it suffices to show that every finitely presented virtually abelian group $V$ has decidable power 
coset membership problem.  
	
The power coset membership problem is a rational subset intersection problem: a coset of a finitely generated subgroup is a rational 
subset, as is a cyclic group minus one element, and the problem asks to determine whether such an intersection is non-empty. 
But the rational subset intersection problem in any finitely-generated group is Turing-equivalent to the rational subset intersection problem in 
any finitely generated finite-index subgroup.  An explicit construction of the reduction is given in \cite{Gru99Thesis} Thm. 2.3.3.  So we reduce 
power coset membership in $V$ to an instance of rational subset intersection in a free abelian finite-index subgroup $A\leq V$, which 
is decidable (see \cite{Gru99Thesis} Prop. 2.2.12).  Note that the reduction requires a coset diagram of $A$ in $V$: such a diagram may be 
constructed since $V$ has decidable membership problem (virtually abelian groups are abelian-by-finite, and an algorithm for the membership problem 
in all polycyclic-by-finite groups is given in \cite{BCR91}).

We conclude from Theorem \ref{Thm:PowerCosetCombination} that $G_{1}$ has decidable power coset membership problem and the 
corresponding graph of groups is benign.  The inductive case is proved similarly, with decidable power coset membership in $G_{i}$ 
provided by induction and noting as above that $G_{i}$ is relatively hyperbolic and locally relatively quasi-convex. 

We prove effective coherence by induction. Since finitely generated virtually abelian groups 
are polycyclic-by-finite, they are effectively coherent by \cite{BCR91} Thm. 3.4. 
%In free abelian groups, constructing a presentation for a subgroup amounts to finding the rank of an integer matrix and is therefore solvable. 
%Then it follows from \cite{BW11} Prop 2.3 that virtually abelian groups are effectively coherent.  
Locally quasi-convex groups are 
also effectively coherent (\cite{Gru99Angle} Prop.~6.1), so the folding algorithm of 
\cite{KWM05} Theorem~5.8 shows that $G_{1}$ is effectively coherent, since the graph of groups for $G_{1}$ is benign.  
By induction, $G_{n}$ is effectively coherent.
\end{proof}

Theorem \ref{Thm:EffectiveCoherenceExtensions} implies that every  
finitely generated group $G$ that is discriminated by $\Gamma$ (torsion-free) is effectively coherent, provided $G$ is 
specified as a subgroup of some $G_{n}$.  However, if $G$ is specified by an abstract presentation 
this does not describe an algorithm unless the embedding $G\hookrightarrow G_{n}$ is known.  
We are now in a position to compute this embedding.  If $\Gamma$ does have torsion, we are not 
aware of an embedding of $G$ into iterated centralizer extensions of $\Gamma$, nor into a chain of the form (\ref{Eqn:ChainElementary}).

In \cite{KM13}, it was 
shown that one may compute a finite collection of homomorphsims from $G$ to groups obtained from $\Gamma$ by centralizer extensions, at least one 
of which must be injective. This was proved for all torsion-free hyperbolic groups $\Gamma$, but 
a method to identify an injective homomorphism from this collection could not be given. With the additional 
assumption of local quasi-convexity, effective coherence of $G_{n}$ allows us to determine which homomorphism is injective.

\begin{theorem}\label{Thm:Embedding}
Let $\Gamma$ be a torsion-free locally quasi-convex hyperbolic group.  There is an algorithm that, given a finitely presented group $G$:
\begin{enumerate}[(i)]
\item if $G$ is discriminated by $\Gamma$, computes a sequence centralizer extensions
\[
\Gamma=G_{0} < G_{1} < \ldots < G_{n}
\]
and an embedding $G\hookrightarrow G_{n}$;
\item if $G$ is not discriminated by $\Gamma$, runs forever.
\end{enumerate}
If in addition a solution to the word problem in $G$ is given, the algorithm terminates when $G$ is not discriminated by $\Gamma$, reporting this fact.
\end{theorem}
\begin{proof}
From Theorem~3.17 of \cite{KM13}, we may effectively construct finitely many groups $H_{1}, \ldots, H_{m}$, each given as sequence centralizer 
extensions of $\Gamma$, and homomorphisms $\phi_{i}: G\rightarrow H_{i}$ such that if $G$ is discriminated by $\Gamma$, at least one $\phi_{i}$ 
is injective.  As observed above, each $H_{i}$ is toral relatively hyperbolic and locally relatively quasi-convex. 

For each $i$, construct a presentation for $\phi_{i}(G)$ using Theorem~\ref{Thm:EffectiveCoherenceExtensions}. 
Since $\phi_{i}(G)\leq H_{i}$ is finitely generated, it is relatively quasi-convex and hence toral relatively hyperbolic 
(by \cite{Hru10} Thm. 9.1 or \cite{MW11Local} Thm. 1.8).
If $G$ is discriminated by $\Gamma$, 
then $G$ is also toral relatively hyperbolic since it is isomorphic to some $\phi_{i}(G)$. 
We check if $G$ is toral relatively hyperbolic using the algorithm given in Theorem~0.2 of \cite{Dah08}, 
which terminates if so and runs forever if not.

If this algorithm terminates, we check for each $i=1,\ldots, m$ whether or not $G$ and $\phi_{i}(G)$ are isomorphic, using the 
solution to the isomorphism problem for toral relatively hyperbolic groups given in \cite{DG08}.  If $G\not\simeq\phi_{i}(G)$, 
then $\phi_{i}$ is not injective.  If $G\simeq \phi_{i}(G)$, then $G$ is discriminated by $\Gamma$ (since it is isomorphic to a subgroup of $H_{i}$,  
which is discriminated by $\Gamma$), 
and hence is Hopfian by \cite{Sel09} Thm. 1.12.
Consequently, $\phi_{i}$ is injective.

If a solution to the word problem in $G$ is given, we may run in parallel the following algorithm.  Enumerate non-trivial elements $g\in G$ and check 
whether or not $\phi_{i}(g)$ is trivial for each $i$.  If $G$ is not discriminated by $\Gamma$, 
all of the $\phi_{i}$ must fail to be injective and the algorithm will eventually find for each $i$ a non-trivial $g_{i}\in G$ such that $\phi_{i}(g_{i})=1$.
\end{proof}

Note that this theorem provides a \emph{recognition algorithm} for groups $G$ discriminated by $\Gamma$, provided the word problem is decidable in 
$G$. It also gives effective coherence of $G$ when $G$ is given by a presentation rather than as a subgroup of $G_{n}$.

\begin{theorem}\label{Thm:EffectiveCoherence}
There is an algorithm that, given 
\begin{itemize}
\item a finitely presented group $\Gamma$ that is hyperbolic, torsion-free, and locally quasi-convex,
\item a finitely presented group $G$ that is discriminated by $\Gamma$, and 
\item a finite subset $X\subset G$, 
\end{itemize}
computes a presentation for the subgroup $\langle X\rangle$ generated by $X$.
\end{theorem}

In the case when $\Gamma$ is a free group, the above result was obtained in \cite{KM06} and in \cite{GW09}.  
Our result provides an alternate proof.

\begin{corollary}
Limit groups are effectively coherent.
\end{corollary}

We may also give an algorithm that enumerates all finitely generated groups discriminated by $\Gamma$.  
%In the case when $\Gamma$ is free, 
%such an enumeration algorithm was used in \cite{GW09} in proving that limit groups are effectively coherent.  

\begin{theorem}\label{Thm:Enumeration}
There is an algorithm that, given a presentation of a group $\Gamma$ that is hyperbolic, torsion-free, and locally quasi-convex, 
enumerates by presentations all finitely generated groups discriminated by $\Gamma$, without repeating isomorphic groups.
\end{theorem}
\begin{proof}
Every finitely generated group discriminated by $\Gamma$ is isomorphic to a subgroup of a group obtained 
from $\Gamma$ by a finite chain of centralizer extensions. 
Enumerate all groups $G_{n}$ obtained from $\Gamma$ by a finite chain of centralizer extensions and all finite subsets $X_{i}\subset G_{n}$.  
Since each 
$G_{n}$ is effectively coherent, we can compute a presentation for $\langle X_{i}\rangle$. Since every $\langle X_{i}\rangle$ is toral relatively 
hyperbolic, we may use the solution to the isomorphism problem from \cite{DG08} to eliminate isomorphic groups.
\end{proof}

\bibliography{quasi-convex}

\bibliographystyle{plain}

\end{document}